\documentclass[12pt, oneside, reqno]{amsart}

\usepackage{amsmath,amsfonts,amssymb,amsthm,amstext,pgf,graphicx,hyperref,verbatim,lmodern,textcomp,color,young,tikz}
\usepackage{relsize}
\usepackage{csquotes}
\numberwithin{equation}{section}
\usepackage{fullpage}
\usepackage[mathscr]{euscript}

\theoremstyle{plain}
\newtheorem{thm}{Theorem}[section]
\newtheorem{lem}[thm]{Lemma}
\newtheorem{prop}[thm]{Proposition}

\theoremstyle{definition}

\newtheorem*{thm*}{Theorem}

\theoremstyle{remark}
\newtheorem*{rem}{Remark}

\newcommand{\om}{\omega}

\newcommand{\Ca}{\mathbb{C}^n}
\newcommand{\Ha}{\mathbb{H}^n}
\newcommand{\N}{\mathbb{N}}
\newcommand{\R}{\mathbb{R}}

\usepackage{color}

\newcommand{\C}{\mathbb C}%
\newcommand{\Z}{\mathbb Z}%
\newcommand{\Bea}{\begin{eqnarray*}}
	\newcommand{\Eea}{\end{eqnarray*}}
\newcommand{\be}{\begin{equation}}
\newcommand{\ee}{\end{equation}}

\title[ Spherical means on Heisenberg group]
{On the lacunary spherical maximal function\\ on the Heisenberg group}

\author[P. Ganguly S. Thangavelu]{Pritam Ganguly and Sundaram Thangavelu}

\address[P. Ganguly, S. Thangavelu]{Department of Mathematics\\
 Indian Institute of Science\\
560 012 Bangalore, India}
\email{pritamg@iisc.ac.in}
\email{veluma@iisc.ac.in}


\thanks{}

\begin{document}
\maketitle

\begin{abstract} In this paper we investigate the $L^p$ boundedness of the lacunary maximal function $ M_{\Ha}^{lac} $ associated to the spherical means $ A_r f$ taken over Koranyi spheres on the Heisenberg group. Closely following  an approach used by  M. Lacey in the Euclidean case, we obtain sparse bounds for these maximal functions leading to new unweighted and weighted estimates. The key ingredients in the proof are the $L^p$ improving property of the operator $A_rf$ and a continuity property of the difference $A_rf-\tau_y A_rf$, where $\tau_yf(x)=f(xy^{-1})$ is the right translation operator. 

\end{abstract}

\section{ Introduction and the main results}
The study of spherical means has recieved considerable attention in the last few decades. In 1976, Stein first considered the spherical maximal function on $\mathbb{R}^n$ defined by 
	$$Mf(x) = \sup_{r>0} |f\ast \mu_r(x)| =  \sup_{r>0}\Big| \int_{|y|=r} f(x-y) d\mu_r(y)\Big|$$ where $\mu_r$ is the normalised measure on the Euclidean sphere of radius $r$. In \cite{Stein}, for $n\geq 3$   he proved that
	$$\|Mf\|_p\leq C \|f\|_p \ \text{if\ and\ only\ if} \ p>\frac{n}{n-1}.$$ Later Cowling-Mauceri \cite{CM} revisited this and proved Stein's result using completely different aruguments. In 1986, Bourgain \cite{Bourgain} settled the case $n=2$.
On the other hand, as proved by C. P. Calderon \cite{C}, the lacunary spherical maximal function 
 $$M_{lac}f(x) = \sup_{j\in\mathbb{Z}} |f\ast \mu_{2^j}(x)|$$
turned out to be bounded on $ L^p(\R^n) $ for all $ 1 < p < \infty.$ Recently, M. Lacey has come up with a new idea to prove these two results and much more. In \cite{Lacey} he has obtained sparse bounds for both $ M $ and $ M_{lac} $ leading to new weighted norm estimates.
 
Our aim in this paper is prove sparse bounds for the lacunary spherical maximal functions  on the Heisenberg group. Recall that  the Heisenberg group $\Ha:=\Ca \times\mathbb{R}$ is equipped with the group operation $$(z, t).(w, s):=\big(z+w, t+s+\frac{1}{2}\Im(z.\bar{w})\big),\ \forall (z,t),(w,s)\in \Ha.$$ 
On this group we have a family of non-isotropic dilations defined by $\delta_r(z,t):=(rz,r^2t), \ \forall (z,t)\in\Ha$ for every $ r > 0.$ These $ \delta_r $ are automorphisms of the group $ \Ha.$ The Koranyi norm of $(z,t)$ in $\Ha$ is defined  by  $$|(z,t)|:=(|z|^4+16t^2)^{\frac{1}{4}}$$
and it is easy to see that $|\delta_r(z,t)|=r|(z,t)|$, i.e., the norm is homogeneous of degree 1 relative to these non-isotropic dilations.  The Haar measure on $ \Ha $ is simply the Lebesgue measure $ dz dt.$  As in the Euclidean case, this Haar measure has a polar decomposition. Let $S_K:=\{(z,t)\in\Ha:|(z,t)|=1\}$ be the unit sphere with respect to the Koranyi norm. Then there is a unique  Radon measure $ \sigma $ on $S_K$ such that for every integrable function on $\Ha$ we have
\begin{align}
 \int_{\Ha}f(z,t)dzdt=\int_{0}^{\infty}\int_{S_K}f(\delta_r\omega)r^{Q-1}d\sigma(\omega)dr 
\end{align}
where $ Q = (2n+2) $ is known as the homogeneous dimension of  $\Ha.$

Let us fix some more notations first. Given $f$  on $\Ha,$ its dilation $\delta_r f $ is simply defined by $ \delta_r f(z,t)=f(\delta_r(z,t)),\ \forall (z,t)\in \Ha.$ 
More generally, if $U$ is a distribution on $\Ha$, then its dilation $\delta_rU$ is defined by $\langle \delta_rU, \phi \rangle :=\langle U, \delta_r \phi\rangle,\ \forall \phi\in C^{\infty}_{c}(\Ha).$ We let $ \sigma_r = \delta_r \sigma $ and define the spherical mean value operator $ A_r $ by 
$$ A_rf(x) = f\ast \sigma_r(x) = \int_{|y|=1} f(x.\delta_r y^{-1}) d\sigma(y).$$
The associated spherical maximal function $ M_{\Ha}f(x) = \sup_{r>0} |A_rf(x)| $ was studied by M. Cowling in \cite{Cowling} where he established the following result.

\begin{thm}  The spherical maximal function $ M_{\Ha} $ is bounded on $ L^p(\Ha) $ for all $ p > \frac{2n+1}{2n}.$
\end{thm}

This is the Heisenberg analogue of Stein's theorem on $ \R^n.$ 
 Also in \cite{VF}, using a square function argument, Fischer proved the boundedness of the spherical maximal operator on functions on free step two nilpotent Lie groups which includes the above result of Cowling. 
In this paper we establish the following analogue of Calderon's theorem for the Heisenberg group. Let
$$ M_{\Ha}^{lac}f(x) = \sup_{k \in \Z} |A_{\delta^k} f(x)| $$ where $ \delta \in (0,1/96) $ is fixed be the lacunary spherical maximal function on $ \Ha.$ We prove:

\begin{thm} 
 \label{thm:spherical}
 Let $ n \geq 2.$ The lacunary spherical maximal function $ M_{\Ha}^{lac}  $ is bounded on $ L^p(\Ha) $ for all $ 1< p < \infty.$
\end{thm}

This is the analogue of  Calderon's theorem for the  spherical  averages  on Koranyi spheres.  We establish the above result by closely following Lacey \cite{Lacey} and proving  a sparse domination of the lacunary spherical maximal function stated below.
Theorem \ref{thm:spherical}, as well as certain weighted versions that are stated in Section 5 are easy consequences of the sparse bound in Theorem \ref{thm:sparse}, which is the main result of this paper. In order to state the result we  need to  set up some notation. 

As in the case of $ \R^n $, there is a notion of dyadic grids on $ \Ha $, the members of which are called (dyadic) cubes.
A collection of cubes $\mathcal{S}$ in $\Ha $ is said to be $ \eta$-sparse if there are sets $\{E_S \subset S:S\in \mathcal{S}\}$ which are pairwise disjoint and satisfy $|E_S|>\eta|S|$ for all $S\in \mathcal{S}$. For any cube $Q$ and $1<p<\infty$, we define
 \[ \langle f\rangle_{Q,p}:=\bigg(\frac{1}{|Q|}\int_Q|f(x)|^pdx\bigg)^{1/p}, \qquad  \langle f\rangle_{Q}:=\frac{1}{|Q|}\int_Q|f(x)|dx. \]
 In the above, $ x =(z,t) \in \Ha $ and $ dx = dz dt $ is the Lebesgue measure on $ \C^n \times \R $, which, as we have already mentioned, is the Haar measure on the Heisenberg group.
Following  Lacey \cite{Lacey}, by the term  $(p,q)$-sparse form we mean the following: 
\[\Lambda_{\mathcal{S},p,q}(f,g)=\sum_{S\in\mathcal{S}}|S|\langle f\rangle_{S,p}\langle g\rangle_{S,q}.\]

\begin{thm} 
\label{thm:sparse}
Assume $ n \geq 2$ and fix $  0<\delta< \frac{1}{96}$. Let $ 1 < p, q < \infty $ be such that $ (\frac{1}{p},\frac{1}{q}) $ belong to the interior of the triangle joining the points $ (0,1), (1,0) $ and $ (\frac{2n}{2n+1},\frac{2n}{2n+1}).$ Then for any  pair of compactly supported bounded functions $ (f,g) $ there exists a $ (p,q)$-sparse form such that $ \langle  M_{\Ha}^{lac}f, g\rangle \leq C \Lambda_{\mathcal{S},p,q}(f,g).$
\end{thm}
We remark that as  in the Euclidean case we can take any $\delta>0$ in Theorem \ref{thm:sparse}, in particular we can choose $\delta=1/2$.

Also on Heisenberg group, Nevo-Thangavelu considered the spherical mean taken over complex spheres $\{(z,0)\in\Ha:|z|=r\}$. In \cite{NeT}, they showed that the corresponding maximal function is bounded on $L^p(\Ha)$ for $p>\frac{2n-1}{2n-2}.$ Later Narayanan-Thangavelu in \cite{NaT} and independently Muller-Seeger in \cite{MS} improved that result and proved that the maximal function is bounded on $L^p(\Ha)$ if and only if $p>\frac{2n}{2n-1}.$ Recently, Bagchi et al \cite{BHRT} have proved the analogue of Calderon's theorem for  the associated lacunary spherical maximal function by obtaining  a sparse bound as in Lacey \cite{Lacey}.

The plan of the paper is as follows. After collecting some relevant results from the harmonic analysis on Heisenberg groups in Section 2 we establish the $ L^p $ improving properties of the spherical averages in Section 3. In Section 4 we study the continuity properties of the same. Finally in Section 5 we sketch the proof of the sparse domination (i.e. proof of Theorem \ref{thm:sparse}) and deduce weighted and unweighted inequalities.

\section{Preliminaries}
\subsection{ Fourier transform on $ \Ha$}
In this section we collect some basic results from the harmonic analysis on Heisenberg groups that will play important roles in the study of the spherical maximal function.  The Heisenberg group $ \Ha$ introduced in the previous section is a nilpotent Lie group which is non-commutative and yet with a very simple representation theory. For each non zero real number $ \lambda $ we have an infinite dimensional representation $ \pi_\lambda $ realised on the Hilbert space $ L^2( \R^n).$ These are explicitly given by
$$ \pi_\lambda(z,t) \varphi(\xi) = e^{i\lambda t} e^{i(x \cdot \xi+ \frac{1}{2}x \cdot y)}\varphi(\xi+y),\,\,\,$$
where $ z = x+iy $ and $ \varphi \in L^2(\R^n).$ These representations are known to be  unitary and irreducible. Moreover, upto unitary equivalence these account for all the infinite dimensional irreducible unitary representations of $ \Ha.$ As the finite dimensional representations of $ \Ha$ do not contribute to the Plancherel measure  we will not describe them here.

The Fourier transform of a function $ f \in L^1(\Ha) $ is the operator valued function obtained by integrating $ f $ against $ \pi_\lambda$:
$$ \hat{f}(\lambda) = \int_{\Ha} f(z,t) \pi_\lambda(z,t)  dz dt .$$  Note that $ \hat{f}(\lambda) $ is a bounded linear operator on $ L^2(\R^n).$ It is known that when $ f \in L^1 \cap L^2(\Ha) $ its Fourier transform  is actually a Hilbert-Schmidt operator and one has
$$ \int_{\Ha} |f(z,t)|^2 dz dt = (2\pi)^{-n-1} \int_{-\infty}^\infty \|\hat{f}(\lambda)\|_{HS}^2 |\lambda|^n d\lambda .$$
The above allows us to extend  the Fourier transform as a unitary operator between $ L^2(\Ha) $ and the Hilbert space of Hilbert-Schmidt operator valued functions  on $ \R $ which are square integrable with respect to the Plancherel measure  $ d\mu(\lambda) = (2\pi)^{-n-1} |\lambda|^n d\lambda.$
\subsection{The Heisenberg Lie algebra} We let $ \mathfrak{h}_n $ stand for the Heisenberg Lie algebra consisting of left invariant vector fields on $ \Ha .$  A  basis for $ \mathfrak{h}_n $ is provided by the $ 2n+1 $ vector fields
$$ X_j = \frac{\partial}{\partial{x_j}}+\frac{1}{2} y_j \frac{\partial}{\partial t}, \,\,Y_j = \frac{\partial}{\partial{y_j}}-\frac{1}{2} x_j \frac{\partial}{\partial t}, \,\, j = 1,2,..., n $$
	and $ T = \frac{\partial}{\partial t}.$  These correspond to certain one parameter subgroups of $ \Ha.$  Recall that given such a subgroup $ \Gamma = \{ \gamma(s): s \in \R \} $ one associates the left invariant vector filed
	$$ Xf(g) = \frac{d}{ds}\big{|}_{s=0}f(g\gamma(s)).$$ Associated to each such $ X $ we also have a right invariant vector field $ \widetilde{X} $ defined by
	$$ \widetilde{X}f(g) = \frac{d}{ds}\big{|}_{s=0}f(\gamma(s)g).$$
 It then follows that any $ (x,y,t) \in \Ha $ can be written as $  (x,y,t) = \exp( x\cdot X+y\cdot Y+t T) $ where $ X = (X_1,....,X_n),$  $ Y = (Y_1,....,Y_n)$ and $ \exp $ is the exponential map taking $ \mathfrak{h}_n $ into $ \Ha.$  From the above definition  right invariant vector fields can be explicitly calculated  \[\tilde{X}_j=\frac{\partial}{\partial x_j}-\frac{1}{2}y_j\frac{\partial}{\partial t}\ \ \ \tilde{Y}_j=\frac{\partial}{\partial y_j}+\frac{1}{2}x_j\frac{\partial}{\partial t} : \] they agree with the left invariant ones at the origin.  The representations $ \pi_\lambda $ of $ \Ha $ give rise to the derived representations $ d\pi_\lambda $ of the Lie algebra $\mathfrak{h}_n.$ These are given by
	$$ d\pi_\lambda(X)\varphi  = \frac{d}{ds}\big{|}_{s=0}\pi_\lambda(\exp sX)\varphi .$$ For reasonable functions $f$ and any right invariant vector field $X$, it is known that $$\pi_{\lambda}(\widetilde{X}f)=d\pi_{\lambda}(\tilde{X})\pi_{\lambda}(f)$$ where $d\pi_{\lambda}$ is the derived representation of the Heisenberg lie algebra corresponding to $\pi_{\lambda}.$ It is also well-known that $d\pi_{\lambda}(\widetilde{X}_j)=i\lambda\xi_j$ and  $d\pi_{\lambda}(\tilde{Y}_j)=\frac{\partial}{\partial\xi_j}.$
 Now writing $\tilde{Z}_j=\tilde{X_j}+i\tilde{Y_j}$,  $\overline{\tilde{Z}_j}=\tilde{X_j}-i\tilde{Y_j}$ and using the above results we have 
 $$\pi_{\lambda}(\tilde{Z}_jf)=iA_j(\lambda)\pi_{\lambda}(f)\ and\ \pi_{\lambda}(\overline{\tilde{Z}_j}f)=iA_j^*(\lambda)\pi_{\lambda}(f)$$ where $A_j(\lambda)$ and $A_j^*(\lambda)$ are the annihilation and creation operators given by $$A_j(\lambda)=\Big(-\frac{\partial}{\partial\xi_j}+i\lambda\xi_j\Big),~~~~A_j^*(\lambda)=\Big(\frac{\partial}{\partial\xi_j}+i\lambda\xi_j\Big).$$ 
 We make use of these relations in the proof of the continuity property of the spherical means, see Section 4.
 
 \subsection{ The measure on the Koranyi sphere} Let $ S_K = \{ (z,t) \in \Ha : |(z,t)| = 1\} $ be the Koranyi sphere of radius $1.$ Then it is well known that there is a Radon measure $ \sigma$ on $ K$ which gives rise to the following polar decomposition of the Haar measure on the Heisenberg group:
 $$  \int_{\Ha} f(x) dx  = \int_0^\infty \big( \int_{S_K} f(\delta_ry) d\sigma(y) \big) r^{Q-1} dr $$ where
 $ Q = 2n+2 $ is known as the homogeneous dimension of $ \Ha.$ For any $ r>0 $ we define the measure $ \sigma_r = \delta_r \sigma $ and  note that it is supported on $  K _r= \{ (z,t) \in \Ha : |(z,t)| = r\}. $ We also have another polar decomposition of the Haar measure given by
 $$ \int_{\Ha} f(z,t) dz dt = \int_{-\infty}^\infty \int_0^\infty \big( \int_{|\omega| =1} f(r\omega,t) d\mu(\omega) \big) r^{2n-1} dr dt $$
 where $ \mu $ is the surface measure on the unit sphere in $ \Ca.$ If we let $ \mu_r $ stand for the surface measure on the sphere $ S_r = \{ (z,0)\in\Ha: |z| =r \} $ and $ \delta_t $ for the Dirac measure on $ \R $ supported at the point $ t$ then the measure $ \mu_{r,t} = \mu_r  \ast \delta_t $ is supported on the set $ S_{r,t} = \{ (z,t)\in \Ha: |z| =r \}.$ The measure $ \sigma_r $ can be expressed in terms of the measures $ \mu_{r,t} $ as follows  (see Faraut-Harzallah \cite{FH}):
 $$  \sigma_r =\frac{\Gamma\big(\frac{n+1}{2}\big)}{\sqrt{\pi}\Gamma\big(\frac{n}{2}\big)} \int_{-\pi /2}^{\pi/2} \mu_{r \sqrt{\cos \theta}, \frac{1}{4}r^2 \sin \theta} \,\,\, (\cos \theta)^{n-1} d \theta.$$

\subsection{ Fourier transforms of radial measures} The unitary group $ U(n) $ has a natural action on $ \Ha $ given by $ k.(z,t) = (k.z,t), k \in U(n) $ which induces an action on functions and measures on the Heisenberg group.  We say that a function $ f $ (measure $ \mu $) is radial if it is invariant under the action of $ U(n).$ It is well known that the subspace of radial functions in $ L^1(\Ha) $ forms a commutative Banach algebra under convolution. So is the the space of finite radial measures on $ \Ha.$  The Fourier transforms of such measures are functions of the Hermite operator $ H(\lambda) = -\Delta+\lambda^2 |x|^2.$ 

In fact, if  $ H(\lambda) = \sum_{k=0}^\infty (2k+n)|\lambda| P_k(\lambda)$  stands for the spectral decomposition of this operator, then for a radial measure $ \mu $ we have
$$ \hat{\mu}(\lambda)  = \sum_{k=0}^\infty  R_k(\lambda, \mu) P_k(\lambda).$$ More explicitly, $P_k(\lambda)$ stands for the orthogonal projection of $L^2(\mathbb{R}^n)$ onto the $k^{th}$ eigenspace spanned by scaled Hermite functions $\Phi^{\lambda}_{\alpha}$ for $|\alpha|=k$. The coefficients $ R_k(\lambda,\mu) $ are explicitly given by
$$ R_k(\lambda,\mu)  =  \frac{k!(n-1)!}{(k+n-1)!} \int_{\Ha}  e^{i\lambda t} \varphi_k^\lambda(z,t) d\mu(z,t). $$ 
In the above formula, $ \varphi_k^\lambda $ are the Laguerre functions of type $ (n-1)$:
$$  \varphi_k^\lambda(z)  = L_k^{n-1}(\frac{1}{2}|\lambda||z|^2) e^{-\frac{1}{4}|\lambda||z|^2} $$ where $L^{n-1}_k$ denotes the Laguerre polynomial of type $(n-1)$. We refer the reader to \cite{T} for the definition and properties. In particular, for the measures $ \sigma_r $ we have
$$ R_k(\lambda, \sigma_r) = \frac{\Gamma\big(\frac{n+1}{2}\big)}{\sqrt{\pi}\Gamma\big(\frac{n}{2}\big)}\frac{k!(n-1)!}{(k+n-1)!} \int_{-\pi /2}^{\pi/2} \varphi_k^\lambda(r \sqrt{\cos \theta}) e^{i \lambda \frac{1}{4}r^2 \sin \theta} \,\,\, (\cos \theta)^{n-1} d \theta.$$
Though the above integral cannot be evaluated in a closed form, it can be used to study the maximal function associated to the spherical means $ f \ast \sigma_r.$ See the work of Fischer \cite{VF} where the spherical maximal function in a slightly general context has been studied.

\section{$L^p$ -improving property of the spherical means}

In this section we prove certain $ L^p-L^q $ bounds for the spherical means operator $ A_r.$  In order to prove the required estimates  we embed $ A_r $ into an analytic family of operators and then appeal to Stein's analytic interpolation theorem. First we obtain the following representation of the measure $ \sigma_r $  as a superposition of certain operators which can be handled easily. In what follows we let 
$ P_t(x) =  c_Q~t ~(t^2+|x|^2)^{-\frac{Q+1}{2}}$ and $I_{\gamma}(x):=C(Q,\gamma)|x|^{-Q+i\gamma},  \ x\in\Ha, t>0, \gamma \in \R $ where $ c_Q $ is defined by the condition
$$  c_Q \int_{\Ha} (1+|x|^2)^{-\frac{Q+1}{2}} dx = 1 $$ and $ C(Q,\gamma) =c_n\Gamma\big(\frac{Q-i\gamma}{4}\big)^2/\Gamma\big(\frac{i\gamma}{2}\big).$ 
In the proof of the next  theorem which gives a representation of $ \sigma_t $ in terms of $ P_t $ and $ I_\gamma $ we make use of the following simple lemma.

\begin{lem}
	\label{lm1}
Let $ F(t) = c_Q (1+e^{2t})^{-\frac{Q+1}{2}} e^{Q t}$  and let $ \hat{F}(\gamma)$ stands for the Euclidean Fourier transform of $ F.$ Then
$$ \sqrt{2\pi}\hat{F}(\gamma)=\frac{\Gamma\big(\frac{Q-i\gamma}{2}\big)\Gamma\big(\frac{1+i\gamma}{2}\big)}{\Gamma\big(\frac{Q}{2}\big) \Gamma\big(\frac{1}{2})\big)}.$$
\end{lem}
\begin{proof} By the definition of the Fourier transform
$$ \hat{F}(\gamma) = \frac{c_Q}{\sqrt{2\pi}} \int_{-\infty}^\infty (1+e^{2t})^{-\frac{Q+1}{2}} e^{Qt} e^{-it\gamma} dt $$
which after the change of variables $ e^t = r $ leads to 
$$ \sqrt{2\pi} \hat{F}(\gamma) = c_Q \int_{0}^\infty (1+r^2)^{-\frac{Q+1}{2}}  r^{Q-1-i\gamma} dr. $$ Another change of variables $ (1-t) = (1+r^2)^{-1} $ converts the above integral into the Beta integral
$$  \frac{1}{2}\int_{0}^{1}(1-t)^{\frac{1+i\gamma}{2}-1}t^{\frac{Q-i\gamma}{2}-1}
	= \frac{1}{2}\frac{\Gamma\big(\frac{Q-i\gamma}{2}\big)\Gamma\big(\frac{1+i\gamma}{2}\big)}{\Gamma\big(\frac{Q+1}{2}\big)}. $$ 
	Consequently we obtain
$$ \sqrt{2\pi} \hat{F}(\gamma) = \frac{1}{2} c_Q\frac{\Gamma\big(\frac{Q-i\gamma}{2}\big)\Gamma\big(\frac{1+i\gamma}{2}\big)}{\Gamma\big(\frac{Q+1}{2}\big)}. $$ 
Observe that, by the definition of $ c_Q $ we have 
$$ 1 = \sqrt{2\pi} \hat{F}(0) = \frac{1}{2} c_Q\frac{\Gamma\big(\frac{Q}{2}\big)\Gamma\big(\frac{1}{2}\big)}{\Gamma\big(\frac{Q+1}{2}\big)}$$ which leads to the conclusion
$  \sqrt{2\pi} \hat{F}(\gamma) = \frac{\Gamma\big(\frac{Q-i\gamma}{2}\big)\Gamma\big(\frac{1+i\gamma}{2}\big)}{\Gamma\big(\frac{Q}{2}\big) \Gamma\big(\frac{1}{2})\big)}$ 
completing the proof.
\end{proof}

The following result is the analogue of a theorem by Cowling and Mauceri \cite{CM} proved in the context of $ \R^n.$ We provide the details in the Heisenberg setting for the convenience of the reader.
\begin{thm}
	 For $t>0$ the following representation  holds in the sense of distributions: 
	 $$\sigma_t=  P_t+ \frac{1}{2\pi}  \int_{-\infty}^{\infty}d(Q,\gamma)t^{-i\gamma}I_{\gamma}~~d\gamma.$$
	
\end{thm} \begin{proof}
	Let $u\in C^{\infty}_{c}(1-\delta<|x|<1+\delta)$. Then by using polar decomposition of the Haar measure, 
	$$\int_{\Ha}u(x)|x|^{-Q+i\gamma}dx=\int_{0}^{\infty}\int_{S_K}u(\delta_r \omega)r^{i\gamma-1}d\sigma(\omega)dr.$$ 
By defining $\bar{u}(r)=\int_{S_K}u(\delta_r \omega)d\sigma(\omega)$ we rewrite the above as
	$$\int_{\Ha}u(x)|x|^{-Q+i\gamma}dx=\int_{0}^{\infty}\bar{u}(r)r^{i\gamma-1}dr$$ But by a change of variables we have $$\int_{0}^{\infty}\bar{u}(r)r^{i\gamma-1}dr=\int_{-\infty}^{\infty}\bar{u}(e^t)e^{i\gamma t}dt.$$ Hence by the Fourier inversion formula  we obtain \begin{align}
	\label{eq1}
	\int_{-\infty}^{\infty} \Big(\int_{\Ha}u(x)|x|^{-Q+i\gamma}dx \Big)d\gamma=2\pi \bar{u}(1)=2\pi\langle \sigma, u\rangle.\end{align}
Now define  $d(Q,\gamma)$ by the requirement
	\begin{align}
	\label{eq2}
	d(Q,\gamma)C(Q,\gamma)=1-\frac{\Gamma\big(\frac{Q-i\gamma}{2}\big)\Gamma\big(\frac{1+i\gamma}{2}\big)}{\Gamma\big(\frac{Q}{2}\big) \Gamma\big(\frac{1}{2})\big)}
	\end{align}
and consider the following equation:
$$ \langle \int_{-\infty}^{\infty}d(Q,\gamma)I_{\gamma}d\gamma, u\rangle=\int_{\Ha} \Big(\int_{-\infty}^{\infty}d(Q,\gamma)C(Q,\gamma)|x|^{-Q+i\gamma}d\gamma \Big) u(x) dx.$$ Now changing the order of the integration and using \ref{eq2} and \ref{eq1} we get
	 \begin{align}
	 \label{eq3}
	 \langle \int_{-\infty}^{\infty}d(Q,\gamma)I_{\gamma}d\gamma, u\rangle=2\pi \langle \sigma, u\rangle -\int_{-\infty}^{\infty} \Big(\int_{\Ha}\frac{\Gamma\big(\frac{Q-i\gamma}{2}\big)\Gamma\big(\frac{1+i\gamma}{2}\big)}{\Gamma\big(\frac{Q}{2}\big) \Gamma\big(\frac{1}{2})\big)} u(x)|x|^{-Q+i\gamma}dx \Big)d\gamma.
	 \end{align}
	 Now we simplify the second integral in the above equation. Using polar decomposition we have 
	 \begin{align*}\int_{-\infty}^{\infty}\int_{\Ha}&\frac{\Gamma\big(\frac{Q-i\gamma}{2}\big)\Gamma\big(\frac{1+i\gamma}{2}\big)}{\Gamma\big(\frac{Q}{2}\big) \Gamma\big(\frac{1}{2})\big)}u(x)|x|^{-Q+i\gamma}dxd\gamma \\
	 &=\int_{-\infty}^{\infty}\int_{S_K}\int_{0}^{\infty}\frac{\Gamma\big(\frac{Q-i\gamma}{2}\big)\Gamma\big(\frac{1+i\gamma}{2}\big)}{\Gamma\big(\frac{Q}{2}\big) \Gamma\big(\frac{1}{2})\big)}r^{i\gamma-1}u(\delta_r\omega)drd\sigma(\omega)d\gamma\\
	 &=\int_{-\infty}^{\infty}\int_{S_K}\int_{-\infty
	 }^{\infty}\frac{\Gamma\big(\frac{Q-i\gamma}{2}\big)\Gamma\big(\frac{1+i\gamma}{2}\big)}{\Gamma\big(\frac{Q}{2}\big) \Gamma\big(\frac{1}{2})\big)}e^{i\gamma t}u(\delta_{e^t}\omega)dtd\sigma(\omega)d\gamma.
	 \end{align*}
	 By Fubini's theorem, changing the order of the integration we obtain 
	 \begin{align*}
	 \int_{-\infty}^{\infty}\int_{\Ha}&\frac{\Gamma\big(\frac{Q-i\gamma}{2}\big)\Gamma\big(\frac{1+i\gamma}{2}\big)}{\Gamma\big(\frac{Q}{2}\big) \Gamma\big(\frac{1}{2})\big)} u(x)|x|^{-Q+i\gamma}dxd\gamma \\
	 &=\int_{-\infty}^{\infty}\int_{S_K}\Big(\int_{-\infty
	 }^{\infty}\frac{\Gamma\big(\frac{Q-i\gamma}{2}\big)\Gamma\big(\frac{1+i\gamma}{2}\big)}{\Gamma\big(\frac{Q}{2}\big) \Gamma\big(\frac{1}{2})\big)}e^{i\gamma t}d\gamma \Big)u(\delta_{e^t}\omega)d\sigma(\omega)dt.
	 \end{align*}
	 We now make use of  Lemma \ref{lm1} and obtain
	 \begin{align*}
	 	\int_{-\infty}^{\infty}\int_{S_K}&\Big(\int_{-\infty
	 	}^{\infty}\frac{\Gamma\big(\frac{Q-i\gamma}{2}\big)\Gamma\big(\frac{1+i\gamma}{2}\big)}{\Gamma\big(\frac{Q}{2}\big) \Gamma\big(\frac{1}{2})\big)}e^{i\gamma t}d\gamma \Big)u(\delta_{e^t}\omega)d\sigma(\omega)dt\\
 	&= \sqrt{2\pi} \int_{-\infty}^{\infty}\int_{S_K}\Big(\int_{-\infty
 	}^{\infty}\hat{F}(\gamma)e^{it\gamma}d\gamma\Big)u(\delta_{e^t}\omega)d\sigma(\omega)dt\\
 &=(2\pi) \int_{-\infty}^{\infty}\int_{S_K}F(t)u(\delta_{e^t}\omega)d\sigma(\omega)dt\\
 &= (2\pi) c_Q  \int_{S_K}\int_{0}^{\infty} (1+r^2)^{-\frac{Q+1}{2}} u(\delta_r\omega)r^{Q-1}drd\sigma(\omega).
	 \end{align*}
	 As the last integral in the above chain of equations is nothing but $ (2\pi) \int_{\Ha} u(x) P_1(x)  dx$ we have proved
	  \begin{align}
	 \label{eq4}
	 \int_{-\infty}^{\infty}\int_{\Ha}&\frac{\Gamma\big(\frac{Q-i\gamma}{2}\big)\Gamma\big(\frac{1+i\gamma}{2}\big)}{\Gamma\big(\frac{Q}{2}\big) \Gamma\big(\frac{1}{2})\big)} u(x)|x|^{-Q+i\gamma}dxd\gamma = (2\pi)  \langle P_1, u\rangle.
	 \end{align}
Combining  \ref{eq3} and \ref{eq4} we obtain the following equality which holds in the sense of distributions: $$\sigma=  P_1+ (2\pi)^{-1}  \int_{-\infty}^{\infty}d(Q,\gamma)I_{\gamma} d\gamma.$$
As $ \sigma_t $ is obtained from $ \sigma $ by dilation, the theorem is proved.
\end{proof}

We would like to embed the spherical means $ A_r $ into an analytic family of operators. As in the Euclidean case, this is achieved by observing that the distributions given by the functions
$$ \phi_{r,\alpha}(x) = 2 \frac{r^{-Q}}{\Gamma(\alpha)} \Big(1-\frac{|x|^2}{r^2} \Big)_+^{\alpha-1} , ~~ \Re(\alpha) > 0 $$
converge to $ \sigma_r $ as $ \alpha \rightarrow 0.$  In the Euclidean case the Fourier transform of $ \phi_{r,\alpha}(x) $ is known explicitly, given in terms of Bessel functions, which allows immediate extension as a homomorphic family of distributions. In the case of the Heisenberg group we do not have a useful formula for the (group) Fourier transform of $\phi_{r,\alpha}.$ Hence we make use of the following representation similar to the one proved for $ \sigma_r $ in the preceding theorem in holomorphically extending the operator $ f \ast \phi_{r,\alpha}.$

\begin{prop}
	\label{lpr}
		Let $r>0$, $Re(\alpha)>0.$  Then for any Schwartz function $ f $ on $\Ha $, we have 
		$$f\ast\phi_{r,\alpha}(x)= 2\frac{r^{-Q}}{\Gamma(\alpha)}\int_{0}^{r}\Big(1-\frac{t^2}{r^2} \Big)^{\alpha-1}t^{Q-1}f\ast P_t(x)dt +\frac{1}{2\pi}\int_{-\infty}^{\infty}d(Q,\gamma)r^{-i\gamma}\frac{\Gamma(\frac{Q-i\gamma }{2})}{\Gamma(\alpha+\frac{Q-i\gamma}{2})}f\ast I_{\gamma}(x)d\gamma.$$
\end{prop}
\begin{proof}
	By definition of convolution on Heisenberg group we have 
	$$f\ast\phi_{r,\alpha}(x)=\int_{\Ha}f(x.y^{-1})\phi_{r,\alpha}(y)dy.$$ As  $\phi_{r,\alpha}$ ir radial, integrating in polar coordinates  we get 
	\begin{align*}
	f\ast\phi_{r,\alpha}(x)=& 2 \frac{r^{-Q}}{\Gamma(\alpha)}\int_{0}^{\infty} \Big( \int_{S_K}f(x.\delta_t\om^{-1})\big(1-\frac{t^2}{r^2}\big)^{\alpha-1}_{+}d\sigma(\om) \Big)  t^{Q-1}\,dt\\
	=& 2 \frac{r^{-Q}}{\Gamma(\alpha)}\int_{0}^{r}\big(1-\frac{t^2}{r^2}\big)^{\alpha-1}t^{Q-1}f\ast\sigma_t(x)dt.
	\end{align*} 
	Making use of the representation
	$$ f\ast\sigma_t=f\ast P_t+ \frac{1}{2\pi} \int_{-\infty}^{\infty}d(Q,\gamma)t^{-i\gamma}f\ast I_{\gamma}d\gamma $$ 
	proved in the previous theorem we get  $  f\ast \phi_{r,\alpha} = S_{r,\alpha} f + T_{r,\alpha} f $ where
	
	$$  S_{r,\alpha} f(x) = 2 \frac{r^{-Q}}{\Gamma(\alpha)}\int_{0}^{r}\Big(1-\frac{t^2}{r^2} \Big)^{\alpha-1}t^{Q-1}f\ast P_t(x)dt $$ and 
	$$ T_{r,\alpha}f(x)  =  \frac{1}{2\pi}\int_{-\infty}^{\infty} \Big( 2 \frac{r^{-Q}}{\Gamma(\alpha)} \int_{0}^{r}\Big(1-\frac{t^2}{r^2} \Big)^{\alpha-1}t^{Q-1} t^{-i\gamma} dt \Big)  d(Q,\gamma) f\ast I_{\gamma}(x) d\gamma .$$
The inner integral can be explicitly calculated: 
$$r^{-Q}\int_{0}^{r}\big(1-\frac{t^2}{r^2}\big)^{\alpha-1}t^{Q-1}t^{-i\gamma}dt
	= r^{-i\gamma}\int_{0}^{1}(1-t^2)^{\alpha-1}t^{Q-1-i\gamma}dt$$
	which reduces to a beta integral and yields 
	$$ 2 \frac{r^{-Q}}{\Gamma(\alpha)} \int_{0}^{r}\Big(1-\frac{t^2}{r^2} \Big)^{\alpha-1}t^{Q-1} t^{-i\gamma} dt =  \frac{\Gamma(\alpha)\Gamma(\frac{Q-i\gamma}{2})}{\Gamma(\alpha+\frac{Q-i\gamma}{2})}.$$
Consequently, we obtain the representation
$$  T_{r, \alpha}f:=  \frac{1}{2\pi} \int_{-\infty}^{\infty}d(Q,\gamma)r^{-i\gamma}\frac{\Gamma(\frac{Q-i\gamma }{2})}{\Gamma(\alpha+\frac{Q-i\gamma}{2})}f\ast I_{\gamma}(x)d\gamma.$$
proving the theorem.

\end{proof}

If we define $ A_{r,\alpha} f = f \ast \phi_{r,\alpha}, $ then by the  above proposition we have $ A_{r,\alpha} f = S_{r,\alpha} f + T_{r,\alpha}f $ where
the above holds under the assumption that $ \Re (\alpha) > 0.$ But both $ S_{r,\alpha} $ and $ T_{r,\alpha} $ have analytic continuation to a larger domain of $ \alpha.$ Indeed, $ T_{r,\alpha} $ can be extended to the whole of $ \C $ as an entire function and $ S_{r,\alpha} $ extends holomorphically to the region $ \Re(\alpha) > -n.$ Thus $ A_{r,\alpha}$ is an analytic family of operators and when $ \alpha $ goes to 0 we recover $  f \ast \sigma_r$.

In order to study the $ L^p$ improving property of the spherical mean value operator $ f \rightarrow f \ast \sigma_r $ we use analytic interpolation. It is enough  to prove an $L^p$- improving property for the operator $T_{r,0}$ by studying the family $ T_{r,\alpha}.$  We shall show that for $\alpha=1+i\beta$, the operator $T_{r,\alpha}$ is bounded from $L^{1+\delta}$ to $L^{\infty}$ for any $\delta>0$ and for some negative value of $\Re(\alpha)$, it is bounded on $L^2(\Ha).$  By a dilation argument, we can assume that $ r =1 $ and hence we deal with $ T_\alpha := T_{1,\alpha}.$ To handle the $ L^2 $ boundedness, we need the following Fourier transform computation.
\begin{prop}
	For $\lambda\neq0$, the Heisenberg group Fourier transform of  the distribution $|.|^{-Q+i\gamma}$ is given by 
	$$\widehat{|.|^{-Q+i\gamma}}(\lambda)= (2\pi)^{n+1}  |\lambda|^{-i\gamma/2} \frac{\Gamma(\frac{i\gamma}{2})}{\Gamma(\frac{Q}{4}-\frac{i\gamma}{4})^2} \sum_{k=0}^{\infty}\frac{\Gamma\big(\frac{2k+n}{2}+\frac{2-i\gamma}{4}\big)}{\Gamma\big(\frac{2k+n}{2}+\frac{2+i\gamma}{4}\big)}P_k(\lambda).$$
\end{prop} 

This has been proved in the work of Cowling and Haagerup \cite{CH}. From the above proposition it is now easy to prove the following:
\begin{prop}
	\label{l22}
		Assume that $n\geq1$. Then for any $\alpha\in\mathbb{C}$ with $Re(\alpha)>-n+\frac{1}{2}$ we have $$\|T_{\alpha}f\|_2\leq C(\Im(\alpha))\|f\|_2 $$ where $C(\Im(\alpha))$ has admissible growth. 
\end{prop}
\begin{proof}
	Note that if we write $a(Q,\gamma):=d(Q,\gamma)C(Q,\gamma)$ we have 
	$$T_{\alpha}f(x)=\int_{-\infty}^{\infty}a(Q,\gamma)r^{-i\gamma}\frac{\Gamma(\frac{Q-i\gamma }{2})}{\Gamma(\alpha+\frac{Q-i\gamma}{2})}f\ast |.|^{-Q+i\gamma}(x)d\gamma .$$ 
It is therefore enough to show that	
$$  \int_{-\infty}^{\infty} |a(Q,\gamma)|   \frac{|\Gamma(\frac{Q-i\gamma }{2})|}{|\Gamma(\alpha+\frac{Q-i\gamma}{2})|} b(Q,\gamma) \,\,d\gamma \leq C(\Im(\alpha)) $$ 
where $ b(Q, \gamma) $ is the norm of the operator $ f \rightarrow f \ast |\cdot|^{-Q+i\gamma} $ on $ L^2(\Ha) $ so that we have the inequality
$$  \| f \ast |\cdot|^{-Q+i\gamma} \|_2  \leq b(Q,\gamma) \|f\|_2.$$
In view of Plancherel theorem for the group Fourier transform on $\Ha$ we have the estimate 
$$  b(Q, \gamma) \leq C \sup_{\lambda} \| \widehat{|.|^{-Q+i\gamma}}(\lambda) \|.$$ 
Using the computation in the previous proposition, we have
$$   \| \widehat{|.|^{-Q+i\gamma}}(\lambda) \| \leq C  \frac{|\Gamma(\frac{i\gamma}{2})|}{|\Gamma(\frac{Q}{4}-\frac{i\gamma}{4})|^2} \sup_{k}\frac{|\Gamma\big(\frac{2k+n}{2}+\frac{2-i\gamma}{4}\big)|}{|\Gamma\big(\frac{2k+n}{2}+\frac{2+i\gamma}{4}\big)|} = C \frac{|\Gamma(\frac{i\gamma}{2})|}{|\Gamma(\frac{Q}{4}-\frac{i\gamma}{4})|^2} .$$ Thus we only need to show that
$$  \int_{-\infty}^{\infty} |a(Q,\gamma)|   \frac{|\Gamma(\frac{Q-i\gamma }{2})|}{|\Gamma(\alpha+\frac{Q-i\gamma}{2})|}  \frac{|\Gamma(\frac{i\gamma}{2})|}{|\Gamma(\frac{Q}{4}-\frac{i\gamma}{4})|^2}  \,\,d\gamma \leq C(\Im(\alpha)) .$$

In order to prove the above, we first recall that $$a(Q,\gamma)=\Big(1- \frac{\Gamma\big(\frac{Q-i\gamma}{2}\big)\Gamma\big(\frac{1+i\gamma}{2}\big)}{\Gamma\big(\frac{Q}{2}\big)\Gamma\big(\frac{1}{2}\big)}\Big)$$ 
and hence $ a(Q,\gamma) $ has a zero at $ \gamma = 0.$ Consequently,
$$  \int_{-1}^{1} \frac{|a(Q,\gamma)|}{|\gamma|}    \frac{|\Gamma(\frac{Q-i\gamma }{2})|}{|\Gamma(\alpha+\frac{Q-i\gamma}{2})|}  \frac{|\Gamma(1+\frac{i\gamma}{2})|}{|\Gamma(\frac{Q}{4}-\frac{i\gamma}{4})|^2}  \,\,d\gamma \leq C_1(\Im(\alpha)) $$  as long as $ \Re(\alpha) > -n-1.$
To prove the integrability away from the origin we make use of the following asymptotic formula for the gamma function: for $ |\nu| $ large
$$\Gamma(\mu+i\nu) \sim \sqrt{2\pi}|\nu|^{\mu-1/2}e^{-\frac{1}{2}\pi|\nu|}.$$ So using this formula, a simple calculation shows that for $ |\gamma| \geq 1$
$$ \frac{|a(Q,\gamma)|}{|\gamma|}    \frac{|\Gamma(\frac{Q-i\gamma }{2})|}{|\Gamma(\alpha+\frac{Q-i\gamma}{2})|}  \frac{|\Gamma(1+\frac{i\gamma}{2})|}{|\Gamma(\frac{Q}{4}-\frac{i\gamma}{4})|^2} \leq C_2(\Im(\alpha)) |\gamma|^{-\Re(\alpha)-n-1/2}.$$ Therefore, it follows that
$$  \int_{|\gamma| \geq 1} \frac{|a(Q,\gamma)|}{|\gamma|}    \frac{|\Gamma(\frac{Q-i\gamma }{2})|}{|\Gamma(\alpha+\frac{Q-i\gamma}{2})|}  \frac{|\Gamma(1+\frac{i\gamma}{2})|}{|\Gamma(\frac{Q}{4}-\frac{i\gamma}{4})|^2}  \,\,d\gamma \leq C_2(\Im(\alpha)) $$  for all $ \Re(\alpha) > -n+1/2.$
This completes the proof of the proposition.
\end{proof}

\begin{prop} \label{l1i}
			For any $\beta\in\mathbb{R}$ and $ p > 1$ we have, $$\|T_{1+i\beta}f\|_{\infty}\leq C(\beta)\|f\|_p.$$
\end{prop} 
  \begin{proof}
  	 For any $p>1$, to prove $L^p\rightarrow L^{\infty}$ estimate, first note that in view of the proposition \ref{lpr}, for any $\beta\in\mathbb{R}$ we have
  	 \begin{align}
  	 T_{1+i\beta}f(x)= f\ast\phi_{1,1+i\beta}(x)-\frac{1}{\Gamma(1+i\beta)}\int_{0}^{1}\Big(1-t^2 \Big)^{i\beta}t^{Q-1}f\ast P_t(x)dt.
  	 \end{align}
  	 So for any $x\in\Ha$ we have 
  	 \begin{align}
  	 \label{rep2}
  	 | T_{1+i\beta}f(x)|\leq |f\ast\phi_{1,1+i\beta}(x)|+ \frac{1}{|\Gamma(1+i\beta)|}\int_{0}^{1} t^{Q-1}|f\ast P_t(x)|dt .
  	 \end{align}
  	 Now see that $$|f\ast\phi_{1,1+i\beta}(x)|\leq \frac{1}{|\Gamma(1+i\beta)|} |f|\ast \chi_{B_1}(x)\leq C(\beta)\|f\|_p.$$
	 As we have $ P_t(x) = t^{-Q}P_1( \delta_t^{-1}x)$ it follows that for any $ p>1$,
	  $$  | f \ast P_t(x)|  \leq t^{-Q} \|f\|_p  \big(\int_{\Ha} P_1(\delta_t^{-1}x)^{p'} dx \big)^{1/p'}  \leq C t^{-Q/p} \|f\|_p .$$ Consequently, the integral in (1.2) is bounded by $  \|f\|_p .$
  	  Finally, these two estimates together with \ref{rep2} we get $$\|T_{1+i\beta}f\|_{\infty}\leq C(\beta) \|f\|_p.$$
  \end{proof}
  
  By using the above two propositions and an analytic interpolation argument, we now prove the following result.
\begin{prop}
	\label{it}
	For any $ 0 < \delta < 1$, we have  $T_0:L^{p_\delta}(\Ha)\rightarrow L^{q_\delta}(\Ha)$ where $ p_\delta = (1+\delta) \frac{2n+1-\delta}{2n} $ and $ q_\delta = (2n+1-\delta).$ 
\end{prop}
\begin{proof}
Given $\delta >0$ we consider the holomorphic family of operators $T_{L(z)}$, where $L(z)= (-n+\frac{1+\delta}{2}) (1-z)+z.$ 
In view of the Propositions \ref{l22} and \ref{l1i}, Stein's interpolation theorem gives
$$T_{L(u)}:L^{p(u)}(\Ha)\rightarrow L^{q(u)}(\Ha)$$ where $\frac{1}{p(u)}=\frac{1-u}{2}+\frac{u}{1+\delta}$ and $\frac{1}{q(u)}= \frac{1}{2}(1-u).$ Solving for  $L(u)=0$ we get $ u = \frac{2n-1-\delta}{2n+1-\delta}$ and simplifying we get $ p(u) = p_\delta = (1+\delta) \frac{2n+1-\delta}{2n} $ and $ q(u) =q_\delta = (2n+1-\delta).$ 

\end{proof}

The following result which follows from the above end point estimates by means of analytic interpolation describes the $ L^p $ improving properties of the spherical mean operator $ A_r.$
\begin{thm}
		Assume that $n\geq1$. Then we have
		$$  \| A_r f \|_q \leq C r^{Q(\frac{1}{q}-\frac{1}{p})} \|f \|_p $$ 
		whenever $(\frac{1}{p},\frac{1}{q})$ lies in the interior of the triangle joining the points $(0,0),(1,1),$ and $(\frac{2n}{2n+1},\frac{1}{2n+1})$ as well as the straight line joining the points $(0,0),(1,1).$
\end{thm}

\begin{proof} An easy calculation shows that $ \delta_r^{-1} A_1 \delta_r  = A_r $ and hence we can assume that $ r =1.$  
With $p_\delta$ and $q_\delta$ as in the proof of theorem \ref{it}, we first show that \begin{align}
	\label{cor2}
	A_r:L^{p_\delta}(\Ha)\rightarrow L^{q_\delta}(\Ha).\end{align}
	Recall that from the Proposition \ref{lpr} we have  $$f\ast\phi_{1,\alpha}(x)=\frac{2}{\Gamma(\alpha)}\int_{0}^{1}(1-t^2 )^{\alpha-1}t^{Q-1}f\ast P_t(x)dt +T_{\alpha}f(x).$$ Note that $A_1$ is recovered from $f\ast \phi_{1,\alpha}$ by letting $\alpha$ tend to $ 0.$ As the first term converges to $ f\ast P_1(x) $ we have the equation
	$ A_1f(x) = f \ast P_1(x) +T_0 f(x).$
Since  $P_1\in L^p$ for any $p \geq 1$ in view of Young's inequality we get  $ \|f \ast P_1\|_{q_\delta} \leq C \|f\|_{p_\delta}.$ Hence using Proposition 2.7 we see that the same is true of $ A_1.$ Since $A_1 $ is bounded on $ L^p(\Ha) $ for any $ 1 \leq p \leq \infty $  Marcinkiewicz interpolation theorem shows $ \| A_1 f\|_q \leq \|f\|_p $ whenever $ (\frac{1}{p},\frac{1}{q}) $ belongs to the triangle $ \Delta_\delta $ with vertices at $ (0,0), (1,1) $ and $ (\frac{2n}{(1+\delta)(2n+1-\delta)}, \frac{1}{(2n+1-\delta)}).$ The theorem is proved by letting $ \delta $ go to $0.$
\end{proof}
\begin{figure}
	\includegraphics[scale=0.15]{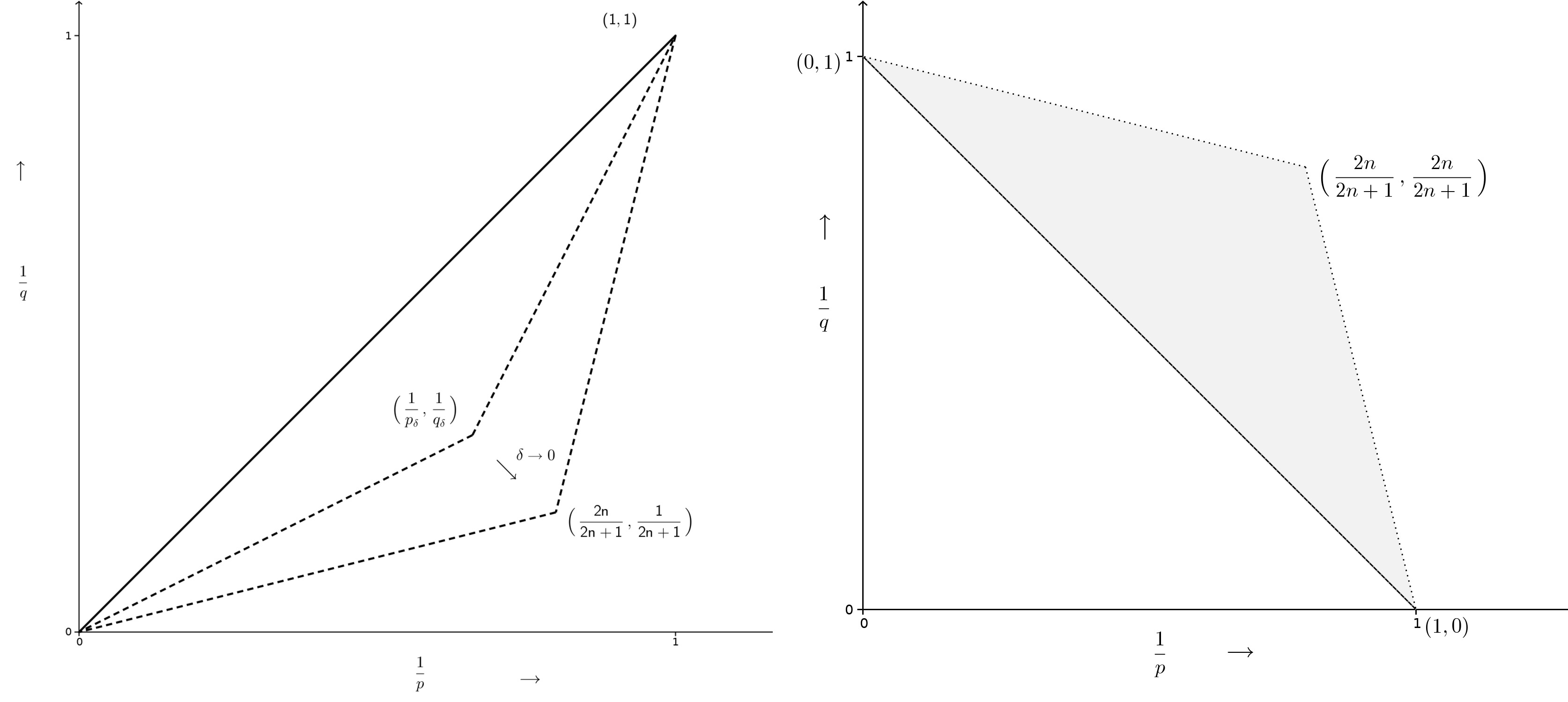}
	\caption{Triangle  $\mathbf{L}_n'$, on the left side, shows the region for $L^p-L^q$ estimates for $A_1$. The dual triangle $\mathbf{L}_n$ is on the right.}
	\label{figurea}
\end{figure}%

\begin{rem}
 It is possible to use the explicit formula for the measure $ \sigma_r $ stated  in the preliminaries  to get $L^p-L^q$ estimates for the spherical means. However, we only get the estimates  for $ (\frac{1}{p},\frac{1}{q}) $ coming from  a smaller triangle. 
\end{rem}

 \section{The continuity property}
 
 Our aim in this section is to prove the following theorem which is known as the continuity property of the spherical mean operator $ A_r.$   For $ {\bf{a}} \in \Ha $ let $ \tau_{\bf{a}}f ( x) = f(x \bf{a}^{-1}) $ be the right translation of $ f $ by $ \bf{a}.$ We prove:
 \begin{thm}
		Assume that $n\geq2$. Then for any ${\bf{a}} \in\Ha$ with $|{\bf{a}}|\leq r$ and for $(1/p,1/q)$ lying in the interior of the triangle joining the points $(0,0),(1,1)$ and $(\frac{2n}{2n+1},\frac{1}{2n+1})$, we have $$\|A_r-A_r\tau_{\bf{a}} \|_{L^p\rightarrow L^q}\leq C r^{Q(1/q-1/p)}  |\delta_{r}^{-1}{\bf{a}}|^{\eta}$$ for some $0<\eta<1.$
\end{thm}

In view of the relation $ A_r-A_r\tau_{\bf{a}}  = \delta_r^{-1} (A_1- A_1 \tau_{\delta_r^{-1}\bf{a}})\delta_r $ we can assume $  r =1 $ without loss of generality. Also in view of Riesz-Thorin interpolation theorem it is enough to prove the case $ p =2.$  It is therefore natural to use the Plancherel theorem for the Fourier transform. Since we do not know the group Fourier transform of $ \sigma_r $ explicitly, we proceed as in the case of the proof of the $ L^p $  improving property by making use of the integral representation for the measure. For the purpose of proving continuity it is convenient to work with the following family rather than $ A_{r,\alpha} $ used earlier.

For $\Re(\alpha)>0$ and $r>0$ we define $$\Phi_{r,\alpha}(x):= 4 \frac{r^{-Q}}{\Gamma(\alpha)}\Big(1-\frac{|x|^4}{r^4}\Big)^{\alpha-1}_{+},\ x\in\Ha$$ 
which converges to $ \sigma_r $ as $ \alpha $ goes to $ 0.$  Given $ {\bf{a}} \in \Ha $ consider the equation
$$ \tau_{\bf{a}}f\ast \Phi_{r,\alpha}(x)-f\ast\Phi_{r,\alpha}(x) =
	\int_{\Ha}f(xy^{-1}{\bf{a}^{-1}})\Phi_{r,\alpha}(y)dy-\int_{\Ha}f(xy^{-1} )\Phi_{r,\alpha}(y)dy.$$
	By making a change of variables and using fundamental theorem of calculus we can rewrite the above as
	$$ \int_{\Ha}f(xy^{-1})\Phi_{r,\alpha}(({\bf{a}})y)dy-\int_{\Ha}f(xy^{-1} )\Phi_{r,\alpha}(y)dy = \int_{\Ha}\Big(\int_{0}^{1}\frac{d}{ds}\Phi_{r,\alpha}((\delta_s{\bf{a}})y)ds\Big) dy.$$
	We are thus led to calculate the derivative of $ \Phi_{r,\alpha}((\delta_s{\bf{a}})y)$ which is given in the following lemma.
	Let $$ X_j = \frac{\partial}{\partial{x_j}}+\frac{1}{2} y_j \frac{\partial}{\partial t}, \,\,Y_j = \frac{\partial}{\partial{y_j}}-\frac{1}{2} x_j \frac{\partial}{\partial t}, \,\, j = 1,2,..., n $$
	and $ T = \frac{\partial}{\partial t}$ be the $ (2n+1) $ left invariant vector fields giving an orthonormal basis for the Heisenberg Lie algebra $ \mathfrak{h}_n.$  It then follows that any $ (x,y,t) \in \Ha $ can be written as $  (x,y,t) = \exp( x\cdot X+y\cdot Y+t T) $ where $ X = (X_1,....,X_n) $ and $ Y = (Y_1,....,Y_n)$ and $ \exp $ is the exponential map taking $ \mathfrak{h}_n $ into $ \Ha.$ We can then check that
	$$  \frac{d}{ds}\big{|}_{s=0} \varphi(\exp(\delta_s(a,b,c)(x,y,t))) = ( a\cdot \widetilde{X} + b \cdot \widetilde{Y}+c T)\varphi(x,y,t) $$
	where  $\widetilde{X} = (\widetilde{X}_1,....,\widetilde{X}_n) $ and $ \widetilde{Y} = (\widetilde{Y}_1,....,\widetilde{Y}_n),$ with
	$$ \widetilde{X}_j = \frac{\partial}{\partial{x_j}}-\frac{1}{2} y_j \frac{\partial}{\partial t}, \,\,\widetilde{Y}_j = \frac{\partial}{\partial{y_j}}+\frac{1}{2} x_j \frac{\partial}{\partial t}.$$ We remark that these are the right invariant vector fields agreeing with the left invariant ones at the origin.

\begin{lem}
	\label{crep}
	Let $  {\bf{a}} =(a,b,c) \in\Ha$ be fixed. Then for any $r>0$ and $\Re(\alpha)>1$ we have  	
	$$ \tau_{\bf{a}}f\ast\Phi_{r,\alpha}(x)-f\ast\Phi_{r,\alpha}(x)= \int_0^1  \Big(\int_{\Ha} \tau_{\delta_s{\bf{a}}}f(xy^{-1})( a\cdot \widetilde{X} + b \cdot \widetilde{Y}+2csT)\Phi_{r,\alpha}(y)dy \Big) ds$$ 

\end{lem}
\begin{proof} In view of the foregoing calculations we have
$$ \tau_{\bf{a}}f\ast \Phi_{r,\alpha}(x)-f\ast\Phi_{r,\alpha}(x) =\int_{\Ha}  f(xy^{-1})\Big(\int_{0}^{1}\frac{d}{ds}\Phi_{r,\alpha}((\delta_s{\bf{a}})y)ds\Big) dy.$$ We note that
$$ \frac{d}{ds}\Phi_{r,\alpha}((\delta_s{\bf{a}})y)= \frac{d}{du}\big{|}_{u=0}\Phi_{r,\alpha}((\delta_{s+u}{\bf{a}})y), $$
$$ (\delta_{s+u}{\bf{a}})y= (ua,ub,(u^2+2su)c)(sa,sb,s^2c)y$$ 
which can also be written in the form  $$ \exp\big(u(a\cdot X+ b\cdot Y)+(u^2+2su)cT\big)(\delta_s{\bf{a}})y.$$
Therefore, taking derivative with respect to $ u $ at $ 0 $ and making use of the calculations done before the lemma we get the result.
\end{proof}

In order to obtain a usable expression for  $( a\cdot \widetilde{X} + b \cdot \widetilde{Y}+2csT)\Phi_{r,\alpha}(y)$ we make use of the following simple lemma.

\begin{lem} Let $ \varphi_0 $ and $ \psi_0$ be smooth functions on $ \R $ and let $ \varphi(x) = \varphi_0(|x|^4) $ and $ \psi(x) = \psi_0(|x|^4) .$ Then for any vector field $ Z $ on $ \Ha$ we have
$ Z\varphi(x) = \frac{\varphi_0^\prime(|x|^4)}{ \psi_0^\prime(|x|^4)}Z\psi(x).$ In particular, 
$$ Z(|x| I_{\gamma}(x)) = \frac{1}{4}(-Q+1+i\gamma) |x|^{-3} Z(|x|^4) I_{\gamma}(x) .$$
\end{lem}

Taking $ \varphi_0(u) = \frac{4r^{-Q} }{\Gamma(\alpha-1)}(1-\frac{u^4}{r^4})_+^{\alpha-1} $ and $ \psi_0(u) = u $ in the above lemma we get the following.
\begin{lem} Assuming that $ \Re{(\alpha)} > 1$ we have 
$$( a\cdot \widetilde{X} + b \cdot \widetilde{Y}+2csT)\Phi_{r,\alpha}(y) = c r^{-3}  \Big(2cs \psi_0(x)+ \sum_{j=1}^n (a_j \varphi_j(x)+b_j \psi_j(x)) \Big) \Phi_{r,\alpha-1}(x) $$
where $ \varphi_j $ and $ \psi_j $ are homogeneous of degree three whereas $ \psi_0$ is of degree two.
\end{lem}

Indeed, we only have to take 
$$ \varphi_j(x) =  \widetilde{X}_j(|x|^4),\,\, \psi_j(x) =  \widetilde{Y}_j(|x|^4),\,\,\psi_0(x) = T(|x|^4)$$
and note that as $ \widetilde{X}_j $ and $\widetilde{Y}_j $ are homogeneous of degree one and $ T $ is homogeneous of degree two these functions have the right homogeneity stated in the lemma.

We want to estimate the norm of $ A_r\tau_{{\bf{a}}}f- A_rf $   which will be obtained as the limit of $ \tau_{\bf{a}}f\ast\Phi_{r,\alpha}(x)-f\ast\Phi_{r,\alpha}(x)$ as $ \alpha $ goes to $0.$ Since the expression in the above lemma is proved under the assumption that $ \Re(\alpha) > 1$ we need to holomorphically extend the integral
$$  \int_{\Ha} g(xy^{-1}) \Phi_{r,\alpha-1}(y) ( a\cdot \widetilde{X} + b \cdot \widetilde{Y}+2csT)(|y|^4)\,\, dy $$ for smaller values of $ \alpha$ before passing to the limit. Assuming $ r =1 $ we consider the operator
$$    \int_{\Ha} f(xy^{-1})   \Phi_{1,\alpha-1}(y)\varphi_j(y) dy.$$ We now obtain the following representation for the above integral which allows holomorphic extension.

\begin{prop} For any $ \alpha $ with $ \Re{(\alpha)} > 1 $ we have
$$ \int_{\Ha}  f (xy^{-1})  \Phi_{1,\alpha-1}(y) \varphi_j(y) dy = S_{\alpha,j}f(x)+T_{\alpha,j}f(x).$$
Here $ T_{\alpha,j} $ is entire and $ S_{\alpha,j}$ has a  holomorphic extension to the region $ \Re{(\alpha)} > -\frac{Q-1}{4}+1.$

\end{prop}  
	
\begin{proof}	Integrating in polar coordinates we have
	\begin{align*}
	\int_{\Ha}&f(xy^{-1})\Phi_{1,\alpha-1}(y)\varphi_j(y)dy\\ &=\frac{4}{\Gamma(\alpha-1)}\int_{0}^{1}\int_{S_K}f(x.\delta_{u}\om^{-1})\big(1-u^4\big)^{\alpha-2} u^{Q+2} \varphi_j(\om) d\sigma(\om)du\\ 
	&= \frac{4}{\Gamma(\alpha-1)}\int_{0}^{1}\big(1-u^4\big)^{\alpha-2}u^{Q+2} f\ast (\varphi_j\sigma)_u (x)du
	\end{align*}
	Making  use of the representation of the representation of $ \sigma_u $ given in Theorem 2.2, we have
	$$ f \ast (\varphi_j\sigma)_u(x)= f \ast P_{u,j}(x)+ \frac{1}{2\pi} \int_{-\infty}^{\infty}d(Q,\gamma)u^{-i\gamma} f\ast I_{\gamma,j}(x)d\gamma $$
	where $ P_{u,j}(x) =P_u(x) \varphi_j(x) |x|^{-3} $ and $ I_{\gamma,j}(x) =I_\gamma(x) \varphi_j(x) |x|^{-3} .$ Using this expression we have
	$$ \int_{\Ha}f(xy^{-1})\Phi_{1,\alpha-1}(y)\varphi_j(y)dy = S_{\alpha,j}f(x) + T_{\alpha,j}f(x) $$
	where $S_{\alpha,j}f$ is given by
	$$ S_{\alpha,j}f(x) = \frac{4}{\Gamma(\alpha-1)}\int_{0}^{1}\big(1-u^4\big)^{\alpha-2}u^{Q+2} f\ast P_{u,j}(x)du$$
	and $T_{\alpha,j}f$ by the integral
	$$T_{\alpha,j}f(x) = \frac{4}{\Gamma(\alpha-1)}\int_{0}^{1}\big(1-u^4\big)^{\alpha-2}u^{Q+2} \Big( \frac{1}{2\pi} \int_{-\infty}^{\infty}d(Q,\gamma)u^{-i\gamma} 
	f\ast I_{\gamma,j}(x)d\gamma \Big) du.$$
Holomorphic extension of $ S_{\alpha,j}f $ is a routine matter: we integrate by parts by writing
$$ S_{\alpha,j}f(x) = \frac{1}{\Gamma(\alpha)}\int_{0}^{1} u^{Q-1} f\ast P_{u,j}(x)  \frac{d}{du} \big(1-u^4\big)^{\alpha-1} du.$$
Under the assumption that $ \Re(\alpha) > 1$ the above leads to the formula
$$ S_{\alpha,j}f(x) = \frac{1}{\Gamma(\alpha)}\int_{0}^{1} \big(1-u^4\big)^{\alpha-1}   \frac{d}{du} \big(u^{Q-1} f\ast P_{u,j}(x)\big) du.$$
The above procedure can be repeated as long as $ \Re{(\alpha)} > -\frac{Q-1}{4}+1.$ For example, one more integration by parts  gives 
$$ S_{\alpha,j}f(x) = \frac{1}{ 4\Gamma(\alpha +1)}\int_{0}^{1} \big(1-u^4\big)^{\alpha}  \frac{d}{du} \Big( u^{-3} \frac{d}{du} \big(u^{Q-1} f\ast P_{u,j}(x)\big) \Big) du.$$
valid for for $ \Re(\alpha) > 0 .$ Having taken care of $ S_{\alpha,j}f $ we now turn our attention towards the other term, namely $ T_{\alpha,j}f.$
Interchanging the order of integration in the defining integral of $ T_{\alpha,j} $ we get
$$T_{\alpha,j}f(x) =  \frac{1}{2\pi} \int_{-\infty}^{\infty}d(Q,\gamma) \Big(\frac{4}{\Gamma(\alpha-1)}\int_{0}^{1}\big(1-u^4\big)^{\alpha-2}u^{Q+2} u^{-i\gamma} 
	 du \Big) f\ast I_{\gamma,j}(x)d\gamma .$$
A simple calculation shows that
$$\frac{4}{\Gamma(\alpha-1)}\int_{0}^{1}\big(1-u^4\big)^{\alpha-2}u^{Q+2} u^{-i\gamma} du = \frac{\Gamma(\frac{Q+3-i\gamma}{4})}{\Gamma(\alpha-1+\frac{Q+3-i\gamma}{4})}.$$
Thus we have proved the representation 
$$T_{\alpha,j}f(x) =  \frac{1}{2\pi} \int_{-\infty}^{\infty}d(Q,\gamma) \frac{\Gamma(\frac{Q+3-i\gamma}{4})}{\Gamma(\alpha-1+\frac{Q+3-i\gamma}{4})} f\ast I_{\gamma,j}(x)d\gamma .$$
Observe that the above is well defined and holomorphic on the whole of $ \mathbb{C} $ as $\frac{1}{\Gamma}$ is an entire function. 
\end{proof}

As we have already remarked, in proving Theorem 3.1 we assume $ p =q=2 $ and $ r =1.$ In view of Lemma 3.2 and Proposition 3.5 we are led to show that the operators $ S_{0,j} $ and $ T_{0,j} $ are bounded on $ L^2(\Ha).$ Note that these operators are defined using the function $ \varphi_j.$ We also need to prove the $ L^2 $ boundedness of operators defined in terms of $ \psi_j$ and $ \psi_0.$ As the proofs are similar we only treat $ S_{0,j} $ and $ T_{0,j} .$

The boundedness of $ S_{0,j} $ is something very easy to check. Recall that 
 $$ S_{0,j}f(x) = \frac{1}{ 4}\int_{0}^{1}  \frac{d}{du} \Big( u^{-3} \frac{d}{du} \big(u^{Q-1} f\ast P_{u,j}(x)\big) \Big) du$$
 $$ = \frac{1}{ 4}\int_{0}^{1}   \frac{d}{du} \big((Q-1)u^{Q-5} f\ast P_{u,j}(x) + u^{Q-4} f \ast \frac{d}{du}P_{u,j}(x) \big) $$
which under the assumption that $ n \geq 2$ reduces to 
$$ S_{0,j}f(x) = \frac{1}{4}\big( (Q-1) f\ast P_{1,j}(x) + f\ast K_j(x) \big) $$
where $ K_j(x) = \frac{d}{du}{|}_{u=1} P_{u,j}(x) $ is an integrable function. It is now clear that $ S_{0,j} $ is bounded on $ L^2(\Ha).$ We use the Fourier transform to prove the $ L^2$ boundedness of $ T_{0,j}.$ The following formula has been proved by Cowling and Haagerup \cite{CH}.

\begin{prop} For all $ 0 < \Re{(s)} < (n+1) $ the Fourier transform of the function $ |x|^{-Q+2s} $ is given by
$$  c_n  |\lambda|^{-s} \frac{\Gamma(s)}{\Gamma(\frac{n+1-s}{2})^2} \sum_{k=0}^\infty \frac{\Gamma(\frac{2k+n+1-s}{2})}{\Gamma(\frac{2k+n+1+s}{2})} P_k(\lambda).$$
\end{prop}

\begin{prop} For all $ n \geq 2$ and $ j =1,2,...,n$ the operator $ T_{0,j} $ is bounded on $ L^2(\Ha).$
\end{prop}
\begin{proof} In view of the representation
$$T_{0,j}f(x) =  \frac{1}{2\pi} \int_{-\infty}^{\infty}d(Q,\gamma) \frac{\Gamma(\frac{Q+3-i\gamma}{4})}{\Gamma(\frac{Q-1-i\gamma}{4})} f\ast I_{\gamma,j}(x)d\gamma $$
we only have to show that
$$  \int_{-\infty}^{\infty}|d(Q,\gamma)| \frac{|\Gamma(\frac{Q+3-i\gamma}{4})|}{|\Gamma(\frac{Q-1-i\gamma}{4})|} \|I_{\gamma,j}\| d\gamma < \infty $$
where $ \|I_{\gamma,j}\| $ stands for the norm of the operator $ f \rightarrow f\ast I_{\gamma,j}$ on $ L^2(\Ha).$ In order to estimate the norm of  this operator we use Plancherel theorem for the Fourier transform, namely
$$  \| f\|_2^2 = (2\pi)^{n-1} \int_{-\infty}^\infty  \|\hat{f}(\lambda)\|_{HS}^2 |\lambda|^n d\lambda$$
where $ \hat{f}(\lambda) = \pi_\lambda(f) $ the Fourier transform of $ f$ at $ \lambda.$ In view of the relation $ \pi_\lambda(f \ast g) = \hat{f}(\lambda)\hat{g}(\lambda) $ it follows that 
$$ \| f \ast I_{\gamma,j} \|_2 \leq  \sup_{\lambda \neq 0} \| \widehat{I_{\gamma,j}}(\lambda)\| $$ where
$ \| \widehat{I_{\gamma,j}}(\lambda)\| $ is the operator norm of $ \widehat{I_{\gamma,j}}(\lambda).$ In order to calculate the Fourier transform of $ I_{\gamma,j} $ we recall that $ I_{\gamma,j}(x) = |x|^{-3} \varphi_j(x) I_{\gamma}(x) $ where $ \varphi_j(x) = \widetilde{X}_j (|x|^4).$  In view of Lemma 3.3 we have
$$  I_{\gamma,j}(x) = \frac{4}{(-Q+1+i\gamma)} \widetilde{X}_j(|x| I_{\gamma}(x)) = \frac{4C(Q,\gamma) }{(-Q+1+i\gamma)} \widetilde{X}_j(|x|^{-Q+1+i\gamma}) .$$
This allows us to calculate the Fourier transform of $ I_{\gamma,j} $ in terms of the Fourier transform of $ |x|^{-Q+1+i\gamma} $ which is known explicitly.

For reasonable functions $ f $ it is known that $ \pi_\lambda( \widetilde{X}_jf) = d\pi_{\lambda}( \widetilde{X}_j)\pi_\lambda(f) $ where $ d\pi_{\lambda}$ is the derived representation of the Lie algebra $ \mathfrak{h}_n$ corresponding to $ \pi_{\lambda}.$  It is also known that $ d\pi_{\lambda}( \widetilde{X}_j)\varphi(\xi) 
= i\lambda \xi_j \varphi(\xi) $ where $ \varphi \in L^2(\R^n).$ Rewriting the above in terms of the annihilation and creation operators $ A_j(\lambda)^\ast = \frac{\partial}{\partial \xi_j}+\lambda \xi_j ,\,\,A_j(\lambda) = -\frac{\partial}{\partial \xi_j}+\lambda \xi_j$
we have
$$ \widehat{I_{\gamma,j}}(\lambda) = \frac{2i C(Q,\gamma)}{ (-Q+1+i\gamma)}( A_j(\lambda)+A_j(\lambda)^\ast) \pi_\lambda(|x|^{-Q+1+i\gamma}).$$
If we let $ H(\lambda) $ stand for the (scaled) Hermite operator $ -\Delta+\lambda^2 |\xi|^2 $ on $ \R^n $ then it is well known that the operators $ A_j(\lambda) H(\lambda)^{-1/2} $ and $ A_j^\ast(\lambda) H(\lambda)^{-1/2} $ are bounded on $ L^2(\R^n).$ Hence we are led to check the boundedness of the operator 
$ H(\lambda)^{1/2} \pi_\lambda(|x|^{-Q+1+i\gamma}).$ From Lemma  we infer that
$$\pi_\lambda(|x|^{-Q+1+i\gamma}) = c_n |\lambda|^{-(1+i\gamma)/2} \frac{\Gamma(\frac{1+i\gamma}{2})}{\Gamma(\frac{Q-(1+i\gamma)}{4})^2}\sum_{k=0}^{\infty}\frac{\Gamma\big(\frac{2k+n}{2}+\frac{1-i\gamma}{4}\big)}{\Gamma\big(\frac{2k+n}{2}+\frac{3+i\gamma}{4}\big)}P_k(\lambda).$$
As $ H(\lambda)^{1/2} P_k(\lambda) = ((2k+n)|\lambda|)^{1/2} P_k(\lambda) $ the operator norm of $ \widehat{I_{\gamma,j}}(\lambda) $ is a constant multiple of
$$ \frac{|C(Q,\gamma)|}{ |(-Q+1+i\gamma)|}   \frac{|\Gamma(\frac{1+i\gamma}{2})|}{|\Gamma(\frac{Q-(1+i\gamma)}{4})|^2}  \,\, \sup_{k \in \N} \Big(  (2k+n)^{1/2}\frac{|\Gamma(\frac{2k+n}{2}+\frac{1-i\gamma}{4})|}{|\Gamma(\frac{2k+n}{2}+\frac{3+i\gamma}{4})|}  \Big).$$
In view of  Stirling's formula for large $k$ we have  $$\Bigg|\frac{\Gamma\big(\frac{2k+n}{2}+\frac{1-i\gamma}{4}\big)}{\Gamma\big(\frac{2k+n}{2}+\frac{3+i\gamma}{4}\big)}\Bigg|=\Bigg|\frac{\Gamma\big(\frac{2k+n}{2}+\frac{1-i\gamma}{4}\big)/\Gamma\big(\frac{2k+n}{2}-\frac{i\gamma}{4}\big)}{\Gamma\big(\frac{2k+n}{2}+\frac{3+i\gamma}{4}\big)/\Gamma\big(\frac{2k+n}{2}+\frac{i\gamma}{4}\big)}\frac{\Gamma\big(\frac{2k+n}{2}-\frac{i\gamma}{4}\big)}{\Gamma\big(\frac{2k+n}{2}+\frac{i\gamma}{4}\big)}\Bigg| \leq C \frac{|\big(\frac{2k+n}{2}-\frac{i\gamma}{4}\big)|^{1/4}}{|\big(\frac{2k+n}{2}+\frac{i\gamma}{4}\big)|^{3/4}}.$$
From the above it is clear that
$$\sup_{k \in \N} \Big(  (2k+n)^{1/2}\frac{|\Gamma(\frac{2k+n}{2}+\frac{1-i\gamma}{4})|}{|\Gamma(\frac{2k+n}{2}+\frac{3+i\gamma}{4})|}  \Big)\leq C $$
and consequently the operator norm of $ \widehat{I_{\gamma,j}}(\lambda) $ is bounded by a constant multiple of
$$   \frac{|C(Q,\gamma)|}{ |(-Q+1+i\gamma)|}   \frac{|\Gamma(\frac{1+i\gamma}{2})|}{|\Gamma(\frac{Q-(1+i\gamma)}{4})|^2}.$$ 

Finally we are left with checking the finiteness of the following integral:
$$  \int_{-\infty}^{\infty}|d(Q,\gamma)| \frac{|\Gamma(\frac{Q+3-i\gamma}{4})|}{|\Gamma(\frac{Q-1-i\gamma}{4})|}  \frac{|C(Q,\gamma)|}{ |(-Q+1+i\gamma)|}   \frac{|\Gamma(\frac{1+i\gamma}{2})|}{|\Gamma(\frac{Q-(1+i\gamma)}{4})|^2}d\gamma .$$
Recall   that 
$$ C(Q,\gamma) d(Q,\gamma) = a(Q,\gamma)=\Big(1- \frac{\Gamma\big(\frac{Q-i\gamma}{2}\big)\Gamma\big(\frac{1+i\gamma}{2}\big)}{\Gamma\big(\frac{Q}{2}\big)\Gamma\big(\frac{1}{2}\big)}\Big)$$ 
We now make use of the following asymptotic formula for the gamma function: for $ |\nu| $ tending to infinity
	$$\Gamma(\mu+i\nu) \sim \sqrt{2\pi}|\nu|^{\mu-1/2}e^{-\frac{1}{2}\pi|\nu|}.$$
	In view of this it follows that $ a(Q,\gamma) $ is a bounded function of $ \gamma$ and
$$  \frac{|\Gamma(\frac{Q+3-i\gamma}{4})|}{|\Gamma(\frac{Q-1-i\gamma}{4})|} \leq C_1 |\gamma|,\,\,\,  \frac{|\Gamma(\frac{1+i\gamma}{2})|}{|\Gamma(\frac{Q-(1+i\gamma)}{4})|^2} \leq C_2 |\gamma|^{-n+\frac{1}{2}} $$as $ |\gamma| $ tends to infinity. Consequently, the integral under consideration converges under the assumption that  $ n \geq 2.$
\end{proof}

\section{Sparse bounds and boundedness of the maximal functions}

Our aim in this section is to sketch a proof of  the sparse bounds for the lacunary spherical maximal function stated in Theorem \ref{thm:sparse}. In doing so we closely follow  Bagchi et al \cite{BHRT} . We  equip $ \Ha $ with a metric induced by the Koranyi norm which makes it a homogeneous space. It is well known that on such spaces there is a well defined notion of dyadic cubes and grids with properties similar to their counter parts in the Euclidean setting. However, we need to be careful with the metric we choose since the group is non-commutative.

Recall that the Koranyi norm on $ \Ha $  is homogeneous of degree one with respect to the non-isotropic dilations. Since we are considering $ f \ast \sigma_r $ it is necessary to work with the left invariant metric $ d_L(x,y) = |x^{-1}y| = d_L(0, x^{-1}y)$ instead of the standard metric $ d(x,y) = |xy^{-1}| = d(0, xy^{-1})$, which is right invariant. The balls and cubes are then defined using $ d_L $. Thus $ B(a,r) = \{ x \in \Ha: |a^{-1}x| < r \}.$ With this definition we note that $ B(a,r) =a\cdot B(0,r) $, a fact which is crucial. This allows us to conclude that when $ f $ is supported in $ B(a,r) $ then $ f \ast \sigma_s $ is supported in $ B(a,r+s).$ Indeed, as the support of $ \sigma_s $ is contained in $ \overline{B(0,s)} $ we see that  $ f \ast \sigma_s $ is supported in $ B(a,r)\cdot\overline{B(0,s)} \subset a\cdot B(0,r)\cdot \overline{B(0,s)} \subset B(a,r+s).$ We have the following result by Hytonen \cite{HK}.
\begin{thm}
\label{thm:homogrid}
Let $\delta\in (0,1)$ with $\delta\le 1/96$. Then there exists a countable set of points $\{z_{\nu}^{k,\alpha} : \nu\in \mathscr{A}_k\}$, $k\in \Z$, $\alpha=1,2,\ldots,N$ and a finite number of dyadic systems $\mathcal{D}^{\alpha}:=\cup_{k\in \Z}\mathcal{D}_k^{\alpha}$, $\mathcal{D}_k^{\alpha}:=\{Q_{\nu}^{k,\alpha}:\nu\in \mathscr{A}_k\}$ such that
\begin{enumerate}
\item For every $\alpha\in \{1,2,\ldots, N\}$ and $k\in \Z$ we have
\begin{itemize}
\item[i)] $\Ha=\cup_{Q\in \mathcal{D}_k^{\alpha}}Q$ (disjoint union).
\item[ii)] $Q,P\in \mathcal{D}^{\alpha}\Rightarrow Q\cap P\in \{\emptyset, P, Q\}$.
\item[iii)] $Q_{\nu}^{k,\alpha}\in \mathcal{D}^{\alpha}\Rightarrow B\big(z_{\nu}^{k,\alpha}, \frac{1}{12}\delta^k\big)\subseteq Q_{\nu}^{k,\alpha}\subseteq B\big(z_{\nu}^{k,\alpha}, 4\delta^k\big)$. In this situation $z_\nu^{k,\alpha} $ is called the center of the cube and the side length $\ell{ (Q_\nu^{k,\alpha})}$ is defined to be $ \delta^k.$
\end{itemize}
\item For every ball $B=B(x,r)$, there exists a cube $Q_B\in \cup_{\alpha}\mathcal{D}^{\alpha}$ such that $B\subseteq Q_B$ and $\ell(Q_B)=\delta^{k-1}$, where $k$ is the unique integer such that $\delta^{k+1}<r\le \delta^k.$ 
\end{enumerate}
\end{thm}

Once we have the above theorem we  can proceed as in Bagchi et al \cite{BHRT} to establish the sparse bounds. The main ingredients in the rather long proof are the $ L^p $ improving property of the spherical means $ A_r $ and their continuity property which we have established already. The proof presented in \cite{BHRT} can be repeated verbatim to get the sparse bounds. We refer the reader to \cite{BHRT} for all the details.

\begin{thm} 
	\label{thm:sparse1}
	Assume $ n \geq 2$ and fix $  0<\delta< \frac{1}{96}$. Let $ 1 < p, q < \infty $ be such that $ (\frac{1}{p},\frac{1}{q}) $ belongs to the interior of the triangle joining the points $ (0,1), (1,0) $ and $ (\frac{2n}{2n+1},\frac{2n}{2n+1}).$ Then for any  pair of compactly supported bounded functions $ (f,g) $ there exists a $ (p,q)$-sparse form such that $ \langle M^{\operatorname{lac}}_{\Ha}f, g\rangle \leq C \Lambda_{\mathcal{S},p,q}(f,g).$
\end{thm}
\textit{Sketch of proof:} First we will reduce the case of getting sparse bounds for $ M^{\operatorname{lac}}_{\Ha}$ to a simpler case. For this we need to set up some notations. For each $k\in\mathbb{Z}$ and a cube $Q$ in $\Ha$ with $l(Q)=\delta^k$ writing $ V_{Q}=\cup_{P\in \mathbb{V}_{Q}}P$ where $\mathbb{V}_{Q}=\{P\in \mathcal{D}^1_{k+3}: B(z_{P},\delta^{k+1})\subseteq Q\}$ we define $$A_{Q}f:=A_{\delta^{k+2}}(f{\bf{1}}_{V_Q}).$$ From these definition it is not hard to see that $$A_{\delta^{k+2}}f\le\sum_{\alpha=1}^N\sum_{Q\in\mathcal{D}_k^{\alpha}}A_Qf$$ and also  whenever $supp(f)\subseteq Q$, $supp(A_Qf)\subseteq Q$. Hence it is enough to prove sparse bound for each of the maximal function defined by 
$$M_{\mathcal{D}^{\alpha}}f=\sup_{Q\in \mathcal{D}^{\alpha}}A_Qf\ \  \text{for each}\ \  \alpha=1,2,\ldots,N.$$ Now we use standard trick to linearize the supremum. Suppose $Q_0$ is a cube in $\mathcal{D}$ and $\mathcal{Q}$ be the collection of all dyadic subcubes of $Q_0.$ We define 
$$
E_Q:=\big\{x\in Q : A_Qf(x)\ge \frac{1}{2} \sup_{P\in \mathcal{Q}} A_Pf(x)\big\}
$$ 
for $Q\in \mathcal{Q}$. Note that for any $ x \in \Ha$ there exists  a $Q \in \mathcal{Q}$ such that  
$$ 
A_Qf(x)\ge \frac{1}{2} \sup_{P\in \mathcal{Q}} A_Pf(x) 
$$  
and hence $ x \in E_Q.$  If we define $B_Q=E_Q\setminus \cup_{Q'\supseteq Q}E_{Q'}$, then $\{B_Q: Q\in \mathcal{Q}\}$ are disjoint and also, $\cup_{Q\in \mathcal{Q}}B_Q=\cup_{Q\in \mathcal{Q}}E_Q$. Observing that $\{B_Q:Q\in \mathcal{Q}\}$, for any $f,g>0$ an easy calculation yields $$\langle\sup_{Q\in \mathcal{Q}}A_Qf,g\rangle\le 2 \sum_{Q\in \mathcal{Q}}\langle A_Qf,g\mathbf{1}_{B_Q}\rangle.$$
Now we will make use of the following lemma which is the main ingredient in proving the sparse bound.
\begin{lem}
	\label{lem:key}
	Let $1<p,q<\infty$ be such that $\big(\frac1p,\frac1q\big)$ in the interior of the triangle joining the points $(0,1), (1,0)$ and $\big(\frac{2n}{2n+1},\frac{2n}{2n+1}\big)$. Let $f=\mathbf{1}_{F}$ and let $ g $ be any bounded function supported in $ Q_0$. Let $C_0>1$ be a constant and let $\mathcal{Q}$ be a collection of dyadic sub-cubes of $Q_0\in \mathcal{D}$ for which the following holds
	\begin{equation}
	\label{eq:keyCondf1}
	\sup_{Q'\in \mathcal{Q}}\sup_{Q: Q'\subset Q\subset Q_0}\frac{\langle f\rangle_{Q,p}}{\langle f\rangle_{Q_0,p}}<C_0.
	\end{equation}
	Then there holds 
	$$
	\sum_{Q\in \mathcal{Q}}\langle A_Qf,g_Q\rangle\lesssim |Q_0|\langle f\rangle_{Q_0,p}\langle g\rangle_{Q_0,q}.
	$$
\end{lem}
 Using the continuity property, we have established in the previous  section, this lemma can be proved using exactly the same argument as in Bagchi et al \cite{BHRT}. We refer the reader to \cite{BHRT} for the proof of this lemma. \\
 
 Now once we have  this lemma, the sparse domination result is immediate. Since $f$ is compactly supported, we may assume that the support of $f$ is contained in some cube $Q_0$. Also we can take $f$ to be non-negative. Then notice that for a fixed dyadic grid $\mathcal{D}$ we have 
  $$M_{\mathcal{D}}f=M_{\mathcal{D}\cap Q_0}f=\sup_{Q\in \mathcal{D}\cap Q_0}|A_Qf|.$$Hence we need to prove a sparse bound for $$\sum_{Q\in \mathcal{D}\cap Q_0}\langle A_Qf,g\mathbf{1}_{B_Q}\rangle.$$ Using the previous lemma we first prove for a special case when $f=\mathbf{1}_{F}$ where $F\subseteq Q_0.$ We consider the following set $$\mathcal{E}_{Q_0}:=\{P \ \text{maximal subcube of} \ Q_0 :\langle f\rangle_{P,p}> 2\langle f\rangle_{Q_0,p}\}$$
   Let $E_{Q_0}=\cup_{P\in \mathcal{E}_{Q_0}}$. For a suitable choice of $c_n>1$ we can arrange $|E_{Q_0}|<\frac12|Q_0|$. We let $F_{Q_0}=Q_0\setminus E_{Q_0}$ so that $|F_{Q_0}|\ge \frac12|Q_0|$. We define
  \begin{equation}
  \label{eq:zero}
  \mathcal{Q}_0=\{Q\in \mathcal{D}\cap Q_0: Q\cap E_{Q_0}=\emptyset\}.
  \end{equation}
  Note that when $Q\in \mathcal{Q}_0$ then $\langle f\rangle_{Q,p}\le 2\langle f\rangle_{Q_0,p}$. For otherwise, if $\langle f\rangle_{Q,p}>2\langle f\rangle_{Q_0,p}$ then there exists $P\in \mathcal{E}_{Q_0}$ such that $P\supset Q$, which is a contradiction. For the same reason, if $Q'\in \mathcal{Q}_0$ and $Q'\subset Q\subset Q_0$ then $\langle f\rangle_{Q,p}\le 2\langle f\rangle_{Q_0,p}$. Thus 
  $$
  \sup_{Q'\in \mathcal{Q}_0}\sup_{Q:Q'\subset Q\subset Q_0}\langle f\rangle_{Q,p}\le 2\langle f\rangle_{Q_0,p}.
  $$
  Note that for any $Q\in \mathcal{D}\cap Q_0$, either $Q\in \mathcal{Q}_0$ or $Q\subset P$ for some $P\in \mathcal{E}_{Q_0}$. Thus 
  $$
  \sum_{Q\in \mathcal{D}\cap Q_0}\langle A_Q f,g\mathbf{1}_{B_Q}\rangle=\sum_{Q\in Q_0}\langle A_Q f,g\mathbf{1}_{B_Q}\rangle+\sum_{P\in \mathcal{E}_{Q_0}}\sum_{Q\subset P}\langle A_Q f,g\mathbf{1}_{B_Q}\rangle
  $$
  for any $Q\in \mathcal{Q}_0$, $Q\subset F_{Q_0}$ and hence
  $$
  \sum_{Q\in \mathcal{Q}_0}\langle A_Q f,g\mathbf{1}_{B_Q}\rangle=\sum_{Q\in \mathcal{Q}_0}\langle A_Q f,g\mathbf{1}_{F_{Q_0}}\mathbf{1}_{B_Q}\rangle.
  $$
  Applying Lemma \ref{lem:key} we obtain
  $$
  \sum_{Q\in \mathcal{Q}_0}\langle A_Q f,g\mathbf{1}_{B_Q}\rangle\le C|Q_0|\langle f\rangle_{Q_0,p} \langle g\mathbf{1}_{F_{Q_0}}\rangle_{Q_0,q}.
  $$
  
  Let $\{P_j\}$ be an enumeration of the cubes in $\mathcal{E}_{Q_0}$. Then the second sum above is given by 
  $$
  \sum_{j=1}^{\infty}\sum_{Q\in P_j\cap \mathcal{D}}\langle A_Q f,g\mathbf{1}_{B_Q}\rangle.
  $$
  For each $j$ we can repeat the above argument recursively. Putting everything together we get a sparse collection $\mathcal{S}$ for which
  \begin{equation}
  \label{eq:toget}
  \sum_{Q\in  \mathcal{D}\cap Q_0}\langle A_Q f,g\mathbf{1}_{B_Q}\rangle\le C\sum_{S\in \mathcal{S}}|S| |\langle f\rangle_{S,p} \langle g\mathbf{1}_{F_{S}}\rangle_{S,q}.
  \end{equation}
  This proves the result when $f=\mathbf{1}_F$.\\
  
  Now we will prove the above result \ref{eq:toget} for any non-negetive, compactly supported bounded function $f$. Here we modify the definition of $\mathcal{E}_{Q_0}$ as follows $$\mathcal{E}_{Q_0}:=\{P \ \text{maximal subcube of} \ Q_0 :\langle f\rangle_{P,p}> 2\langle f\rangle_{Q_0,p}\ \text{or}~\langle g\rangle_{P,p}> 2\langle g\rangle_{Q_0,p}\}$$ Now defining $E_{Q_0},F_{Q_0}$ and $\mathcal{Q}_0$ same as before, we have 
  $$
  \sup_{Q'\in \mathcal{Q}_0}\sup_{Q:Q'\subset Q\subset Q_0}\langle f\rangle_{Q,p}\le 2\langle f\rangle_{Q_0,p} ~\text{and}~  \sup_{Q'\in \mathcal{Q}_0}\sup_{Q:Q'\subset Q\subset Q_0}\langle g\rangle_{Q,q}\le 2\langle g\rangle_{Q_0,q}.
  $$
  Now  we will decompose $f$ in the following way so that we can use the sparse domination already proved for characteristic functions. We write $f=\sum_{m}f_m$ where $f_m:=f\mathbf{1}_{E_m}$ and $E_m:=\{x\in Q_0:2^m\le f(x)<2^{m+1}\}.$ For each $m$ applying the sparse domination to $\mathbf{1}_{E_m}$, we get a sparse family $\mathcal{S}_m$ such that $$\sum_{Q\in Q_0\cap \mathcal{D}}\langle A_Q\mathbf{1}_{E_m},g\mathbf{1}_{F_{Q_0}}\mathbf{1}_{B_Q}\rangle\le C\sum_{S\in \mathcal{S}_m}|S|\langle \mathbf{1}_{E_m}\rangle_{S,p}\langle g\mathbf{1}_{F_{Q_0}}\rangle_{S,q}$$ Hence using the fact that for any $Q\in \mathcal{Q}_0$, $Q\subset F_{Q_0}$ we have $$\sum_{Q\in \mathcal{Q}_0}\langle A_Q f_m,g\mathbf{1}_{B_Q}\rangle \le C2^{m+1}\sum_{S\in \mathcal{S}_m}|S|\langle \mathbf{1}_{E_m}\rangle_{S,p}\langle g\mathbf{1}_{F_{Q_0}}\rangle_{S,q}.$$ Now it is clear that if $S\cap F_{Q_0}=\phi$ then $\langle g\mathbf{1}_{F_{Q_0}}\rangle_{S,q}=0.$
  Also when $S\cap F_{Q_0}\neq\phi$, using the definition of $\mathcal{E}_{Q_0}$ and $\mathcal{Q}_0$ one can easily show that $\langle g\mathbf{1}_{F_{Q_0}}\rangle_{S,q}\le C \langle g\rangle_{Q_0,q}.$ 
  Hence we make use of the following lemma proved in \cite{Lacey} ( see also  \cite{BHRT}):
  \begin{lem}
  	\label{lem:carleson}
  	Let $\mathcal{S}$ be a collection of sparse sub-cubes  of a fixed dyadic cube $Q_0$ and let $1\le s<t<\infty$. Then, for a bounded function $\phi$,
  	$$
  	\sum_{Q\in \mathcal{S}}\langle \phi\rangle_{Q,s}|Q|\lesssim \langle \phi\rangle_{Q_0,t}|Q_0|.
  	$$
  \end{lem}
 For some $\rho_1>p$ we have 
 $$
 \sum_{Q\in \mathcal{Q}_0}\langle A_Q f_m,g\mathbf{1}_{B_Q}\rangle\le C2^{m+1}\langle g\rangle_{Q_0,q}\langle \mathbf{1}_{E_m}\rangle_{Q_0,\rho_1}|Q_0|.
 $$ But we know that $f=\sum_mf_m.$ So, finally we obtain
 $$
 \sum_{Q\in \mathcal{Q}_0}\langle A_Q f,g\mathbf{1}_{B_Q}\rangle\le C\langle g\rangle_{Q_0,q}|Q_0|\sum_m2^m\langle \mathbf{1}_{E_m}\rangle_{Q_0,\rho_1}.
 $$
Note that a simple calculation yields $\sum_m2^m\langle \mathbf{1}_{E_m}\rangle_{Q_0,\rho_1}\le C\|f\|_{L^{\rho_1,1}(Q_0,\frac{1}{|Q_0|}dx)}$ where $\|.\|_{L^{\rho_1,1}}$ denotes the Lorentz space norm. Also it is a well-known fact that  for any $\rho>\rho_1$, $L^{\rho_1,1}(Q_0,\frac{1}{|Q_0|}dx)$- norm is dominated by the $L^{\rho}(Q_0,\frac{1}{|Q_0|}dx)$- norm. Hence we have $$
\sum_{Q\in \mathcal{Q}_0}\langle A_Q f,g\mathbf{1}_{B_Q}\rangle\le C\langle g\rangle_{Q_0,q}|Q_0|\langle f\rangle_{Q_0,\rho}.
$$ Now proceeding same as in case $f=\mathbf{1}_F$, we get the following sparse domination $$
\langle M_{\mathcal{D}}f,g\rangle\le  C\sum_{S\in \mathcal{S}}|S| |\langle f\rangle_{S,\rho} \langle g\rangle_{S,q}.
$$ which proves the theorem. 
\qed

As a consequences of the sparse bound we get some new weighted and unweighted inequalities for the lacunary maximal function under consideration.  We now use the above sparse domination along with the following well-known boundedness property of the sparse forms proved in \cite{CUMP}:
\begin{prop}
	
	Let $1\le r<s'\le\infty$. Then,
	$$
	\Lambda_{r,s}(f,g)\lesssim \|f\|_{L^p}\|g\|_{L^{p'}}, \quad r<p<s'.
	$$
\end{prop}
We have thus proved the following result which is the main theorem in this article.
\begin{thm}
	 For $ n \geq 2,$ the lacunary spherical maximal function $ M_{\Ha}^{lac}  $ is bounded on $ L^p(\Ha) $ for all $ 1< p < \infty.$
\end{thm}
Weighted norm inequalities are also well studied in the literature. To state the weighted boundedness properties of the sparse form we need to mention the following terminologies about classes of weights. A weight $w$ is a non-negative locally integrable function defined on $\Ha$.  Given $1<p<\infty$, the Muckhenhoupt class of weights $A_p$ consists of all $w$ satisfying
$$
[w]_{A_p}:=\sup_Q \langle w\rangle_{Q}\langle w^{1-p'} \rangle_{Q}^{p-1}<\infty
$$
where the supremum is taken over all cubes $Q$ in $\Ha$.  On the other hand, a weight $w$ is in the reverse H\"older class $\operatorname{RH}_p$, $1\le p<\infty$, if
$$
[w]_{\operatorname{RH}_p}=\sup_Q\langle w\rangle_Q^{-1}\langle w\rangle_{Q,p}<\infty,
$$
again the supremum taken over all cubes in $\Ha$. The following weighted inequality for sparse form has been proved in \cite{BFP}.
\begin{prop}
	\label{thm:BFP}
	Let $1\le p_0<q_0'\le\infty$. Then,
	$$
	\Lambda_{p_0,q_0}(f,g)\le \{[w]_{A_{p/p_0}}\cdot[w]_{\operatorname{RH}_{(q_0'/p)'}}\}^{\alpha}\|f\|_{L^p(w)}\|g\|_{L^{p'}(\sigma)}, \quad p_0<p<q_0',
	$$
	with $\alpha=\max\Big\{\frac{1}{p-1},\frac{q_0'-1}{q_0'-p}\Big\}$. 
\end{prop} Using this result we have the following weighted boundedness property for the lacunary maximal function:
\begin{thm}
	Let $n\ge2$ and define
	$$
	\frac{1}{\phi(1/p_0)}=\begin{cases}1-\frac{1}{2np_0}, \quad 0<\frac1p_0\le \frac{2n}{2n+1},\\
	2n\Big(1-\frac1p_0\Big), \quad \frac{2n}{2n+1}<\frac1p_0<1.
	\end{cases}
	$$
	Then $M^{lac}_{\Ha}$ is bounded on $L^p(w)$ for $w\in A_{p/p_0}\cap \operatorname{RH}_{(\phi(1/p_0)'/p)'}$ and all $1<p_0<p<(\phi(1/p_0))'$.
\end{thm}

\end{document}